\documentclass[11pt]{amsart}
\usepackage{amsmath}
\usepackage{amssymb,mathrsfs, xcolor, geometry}
\usepackage{hyperref}
\hypersetup{linkbordercolor = {white}, citebordercolor = {white},}
\usepackage[alphabetic]{amsrefs}

\usepackage[alphabetic]{amsrefs}

\numberwithin{equation}{section}

\theoremstyle{definition}
\newtheorem{thm}[equation]{Theorem}
\newtheorem{Thm}[equation]{Theorem}
\theoremstyle{definition}

\newtheorem{prop}[equation]{Proposition} 
\newtheorem{Pro}[equation]{Proposition} 

\newtheorem{Lem}[equation]{Lemma}

\newtheorem{Assumption}[equation]{Assumption}

\newcommand{\ovl}[1]{\overline{#1}}

\newcommand{\pair}[1]{\langle #1\rangle}
\newcommand{\set}[1]{\left\{#1\right\}}
\newcommand{\ppair}[1]{\left\langle #1\right\rangle}
\newcommand{\pr}[1]{\left(#1\right)}
\newcommand{\abs}[1]{\left|#1\right|}
\newcommand{\bb}[1]{\mathbb{#1}} \newcommand{\td}[1]{\widetilde{#1}}

\newcommand{\RR}{\mathbb R}

\newcommand{\D}{\Delta}

\newcommand{\n}{\nabla}
\newcommand{\ep}{\varepsilon}

\newcommand{\dL}{\mathcal L}
\newcommand{\sbst}{\subseteq}
\newcommand{\h}{\textbf H}
\newcommand{\na}{\nabla}
\newcommand{\la}{\langle}
\newcommand{\ra}{\rangle}

\newcommand{\bd}{\partial}

\newcommand{\supp}{\text{supp}}

\newcommand{\N}{\textbf{n}}
\newcommand{\p}{\mathbf p}

\newcommand{\tf}{\tilde{f}}

\DeclareMathOperator{\dist}{dist}

\newcommand{\eps}{\varepsilon}

\newcommand{\bn}{\mathbf{n}}

\newcommand{\cH}{\mathcal{H}}

\title{Closed mean curvature flows with prescribed tangent flows}

\author{Jingwen Chen}
\author{Tang-Kai Lee}
\author{Ao Sun}
\author{Xinrui Zhao}

\email{jingwchmath@gmail.com, aos223@lehigh.edu, 
leetk@math.columbia.edu, and xinrui.zhao@yale.edu}

\date{\today}

\begin{document}
\begin{abstract}
Given an embedded shrinker $\Sigma$ in $\mathbb{R}^{n+1}$ that is either closed, asymptotically conical, or a Cartesian product of such a shrinker with $\mathbb{R}^k$, we construct a closed embedded mean curvature flow whose tangent flow at the first singularity is modeled on $\Sigma$. We also prescribe the first-order asymptotics of the tangent flow. This result is a consequence of a more general theorem that allows us to construct mean curvature flows with an additional force whose tangent flow and first-order asymptotics at the first singularity are prescribed.
\end{abstract}
\maketitle

\section{Introduction}

Singularity formation is one of the central topics in the study of geometric evolution equations. For mean curvature flow, Huisken's monotonicity formula and parabolic rescaling show that tangent flows at singularities are self-similarly shrinking solutions, called (self-)shrinkers. A central problem is to determine which shrinkers can actually arise as singularity models for flows starting from closed hypersurfaces.

Throughout the paper, we consider codimension-one mean curvature flows. A smooth one-parameter family of hypersurfaces $M^n(t) \subset \mathbb{R}^{n+1},\ t \in I,$ evolves by its mean curvature if 
\[
\partial_t x=\vec H(x),
\]
where $\vec H(x)$ is the mean curvature vector of the hypersurface at the point $x$. 
A shrinker is a smooth embedded hypersurface $\Sigma^n\subset \RR^{n+1}$ satisfying
\begin{align}\label{shrinker-eq}
    \vec H_\Sigma+\frac{x^\perp}{2}=0.
\end{align}
Equivalently, the family of hypersurfaces $(\sqrt{-t}\Sigma)_{t<0}$ is a self-similarly shrinking solution to mean curvature flow. If $(x_0,T)$ is a singular point of a mean curvature flow, we say that the singularity $(x_0,T)$ is modeled by $\Sigma$ if the tangent flow is the multiplicity-one self-shrinking flow $(\sqrt{-t}\Sigma)_{t<0}$. We will say $(\sqrt{-t}\Sigma)_{t<0}$ is the multiplicity-one self-shrinking flow associated to $\Sigma$.

The purpose of this paper is to prove realization and asymptotic prescription results for a large class of smooth embedded shrinkers. The classes of models considered here are the following:
\begin{equation}\label{eq:list-shrinkers}
\tag{$\dagger$}
\begin{minipage}{0.82\textwidth}
\begin{enumerate}
    \item closed shrinkers;
    \item smoothly asymptotically conical shrinkers;
    \item Euclidean products $\RR^\ell\times \Sigma_0$, where $\Sigma_0$ is as in (1) or (2).
\end{enumerate}
\end{minipage}
\end{equation}
These classes include the standard families of smooth embedded shrinkers appearing in the literature. In dimension two, they are also consistent with Ilmanen's conjectural picture for smooth embedded shrinkers; see, for example, \cite{I95}*{Note on page 19} and \cite{W14,W16}. 
There are numerous examples of shrinkers in the list \eqref{eq:list-shrinkers} constructed in $\RR^3$, cf. \cite{A89, KKM, N14, KM23, BNS25, IW, K24, SZ25, HMW25} and the references therein. 
We also want to mention that the end behaviors of the shrinkers constructed in \cite{Ke, SWZ, CSgenus} and the low-genus examples in \cite{BNS25, K24} are not explicit.
Nevertheless, they might still belong to \eqref{eq:list-shrinkers}.

Our first main result says that every nonflat shrinker in \eqref{eq:list-shrinkers} can be realized as the tangent flow of a first-time singularity of a closed embedded mean curvature flow.

\begin{thm}\label{thm:euclidean-realization}
Let $\Sigma^n\subset \RR^{n+1}$ be a nonflat smooth embedded shrinker belonging to one of the classes in \eqref{eq:list-shrinkers}. Then there exist a smooth closed embedded hypersurface $M_0^n\subset \RR^{n+1}$, a time $T<\infty$, and a point $x_0\in\RR^{n+1}$ such that the mean curvature flow starting from $M_0$ is smooth on $[0,T)$, has first singular time $T$, and has a tangent flow at $(x_0,T)$ equal to the multiplicity-one self-shrinking flow associated to $\Sigma$.
\end{thm}

Thus, within the class \eqref{eq:list-shrinkers}, there is no further obstruction to realizing a smooth embedded shrinker as a first-time singularity model of a closed embedded mean curvature flow. In contrast, some other singularity models are known to be obstructed from appearing as the blow-up limit of closed geometric flows. For example, the $1$-dimensional embedded translator, grim reaper, can not show up as the blow-up limit of smooth embedded curve shortening flow; the $2$-dimensional Cigar soliton $\times\RR$ can not show up as the blow-up limit of $3$-dimensional Ricci flow of closed Riemannian manifolds.

Theorem \ref{thm:euclidean-realization} is obtained from a more flexible statement for forced mean curvature flow. After a tangential reparametrization, we write such a flow in the form
\[
    \partial_t F=\vec H_F+\phi[F,t]\bn_F,
\]
where $\bn_F$ is a choice of unit normal and $\phi$ is a prescribed scalar forcing term. The precise hypothesis on admissible $\phi$ is stated in Assumption \ref{t:rmcf}.

\begin{thm}[Forced realization]\label{thm:forced-realization}
Let $\Sigma^n\subset \RR^{n+1}$ be either a closed shrinker or a nonflat smoothly asymptotically conical shrinker. 
If the forcing term $\phi$ satisfies Assumption \ref{t:rmcf} with $\delta_0,\eps_0$ in Assumption \ref{t:rmcf} sufficiently small, then there exists a smooth closed embedded hypersurface $M_0\subset \RR^{n+1}$ whose forced mean curvature flow develops a first-time singularity with tangent flow equal to the multiplicity-one self-shrinking flow associated to $\Sigma$.
\end{thm}

Theorem \ref{thm:forced-realization} is useful because several geometric situations become forced mean curvature flows after a natural rescaling or quotient. The first application is to mean curvature flow in a general Riemannian ambient manifold.

\begin{thm}[Riemannian ambient manifolds]\label{thm:riemannian-realization}
Let $(N^{n+1},g)$ be a smooth Riemannian manifold. Let $p\in N$, and identify $T_pN$ with $\RR^{n+1}$ using a normal coordinate centered at $p$. Let $\Sigma^n\subset T_pN$ be a nonflat smooth embedded shrinker of type (1) or (2) in \eqref{eq:list-shrinkers}. Then there exists a smooth closed embedded hypersurface $M_0\subset N$ whose mean curvature flow has first singular time $T<\infty$ and has $p$ as a singular point at time $T$, such that, in normal coordinates centered at $p$, a tangent flow at $(p,T)$ is the multiplicity-one self-shrinking flow associated to $\Sigma$.
\end{thm}

The second application concerns rotationally symmetric flows. Write
\[
    \RR^{n+1}=\RR^k\times \RR^{n-k+1},
\]
and let $SO(n-k+1)$ act on the second factor. The quotient can be identified with the half-space
\[
    \RR^k\times [0,\infty).
\]
An $SO(n-k+1)$-invariant mean curvature flow in $\RR^{n+1}$ induces a forced mean curvature flow for its profile hypersurfaces in this quotient.

\begin{thm}[Rotationally symmetric realization]\label{thm:rotational-realization}
Let $0< k<n$, and let $\Sigma^k\subset \RR^{k+1}$ be a nonflat smooth embedded shrinker of type (1) or (2) in \eqref{eq:list-shrinkers}. Then there exists a smooth closed embedded hypersurface
\[
    M_0^n\subset \RR^{n+1}=\RR^k\times \RR^{n-k+1},
\]
invariant under the standard $SO(n-k+1)$-action, such that the induced profile flow in $\RR^k\times(0,\infty)$ develops a first-time singularity whose tangent flow is modeled on $\Sigma$. Equivalently, the corresponding singularity of the $SO(n-k+1)$-invariant flow in $\RR^{n+1}$ has a tangent flow modeled on
\[
    \Sigma\times \RR^{n-k}.
\]
\end{thm}

Theorems \ref{thm:forced-realization}, \ref{thm:riemannian-realization}, and \ref{thm:rotational-realization} imply Theorem \ref{thm:euclidean-realization} for all shrinkers in \eqref{eq:list-shrinkers}.

We would like to compare the construction method in this paper with the previous work of the second-named and fourth-named authors \cite{LZ}. The construction in \cite{LZ} used a doubling argument, which relies on the symmetry of the ambient Euclidean space, and it produces two identical first-time singularities. In contrast, the present argument in this paper does not rely on a symmetry that produces two identical first-time singularities. Although we still close off a large compact piece of the shrinker, the dynamical selection argument is local near a single prescribed model and yields either a single prescribed first-time singularity or an orbit of such singularities.

We next turn to a finer question. Once a tangent flow has been prescribed, one may ask whether the leading-order asymptotics of the rescaled flow near the singularity can also be prescribed. Let
\[
    L_\Sigma
    =
    \Delta_\Sigma
    -\frac12\langle x,\nabla_\Sigma(\cdot)\rangle
    +|A_\Sigma|^2+\frac12
\]
be the linearized operator of the rescaled mean curvature flow at $\Sigma$ (see Section \ref{prel: shrinker and linear} for a detailed description), acting on the weighted space
\[
    L_W^2(\Sigma)
    =
    L^2\!\left(\Sigma,e^{-|x|^2/4}\,d\mu_\Sigma\right).
\]
With this convention, a Fourier mode $\varphi$ satisfying
\[
    L_\Sigma\varphi=-\lambda\varphi,\qquad \lambda>0,
\]
corresponds formally to a decaying term $e^{-\lambda\tau}\varphi$ in the rescaled flow.

The next theorem shows that such leading-order terms can be prescribed when $\lambda\in(0,1/2)$.
In the asymptotically conical case, the admissibility condition includes the decay range required for the large-scale estimates.

\begin{thm}[Prescribed first-order asymptotics]\label{thm:first-order}
Let $\Sigma^n\subset\mathbb R^{n+1}$ be either a closed self-shrinker or a smoothly asymptotically conical self-shrinker.  Let $\lambda\in(0,1/2)$ be an eigenvalue of $-L_\Sigma$, and let $\varphi$ be an eigenfunction satisfying
\[
    L_\Sigma\varphi=-\lambda\varphi .
\]
Then there exists a mean curvature flow with additional forces satisfying Assumption \ref{t:rmcf} of closed embedded hypersurfaces whose rescaled flow converges to
$\Sigma$ as $\tau\to\infty$, smoothly in the closed case and locally
smoothly in the asymptotically conical case. Furthermore, for all sufficiently large $\tau$, the rescaled flow is represented, on the relevant graphical region, by normal graphs $u(\cdot,\tau)$ over $\Sigma$, and
\[
    \bigl\|
        \chi_\tau u(\cdot,\tau)-e^{-\lambda\tau}\varphi
    \bigr\|_{L_W^2}
    =
    o(e^{-\lambda\tau})
    \qquad\text{as }\tau\to\infty .
\]
Here $\chi_\tau\equiv 1$ in the closed case, while in the asymptotically conical case $\chi_\tau$ is the cutoff function specified in Step 1 of the proof of Theorem \ref{t:prescribe-first-asymptotic}.
\end{thm}

Prescribing the asymptotics near the singularity model is closely related to the stable/unstable manifold theory. In the context of mean curvature flow, Angenent--Velazquez \cite{AV97} used the Wa\.{z}ewski Box principle to construct rotationally symmetric mean curvature flow profile near a cylindrical singularity. Sesum \cite{Sesum08} used the inverse function theory to prescribe the first-order asymptotics of mean curvature flow near a spherical singularity, and as a consequence, proved that the level set function of the mean curvature flow with a spherical singularity may not have $C^3$-regularity. This idea was later generalized in \cite{Strehlke20, SWX2}. For a general closed shrinker, in the $1$-dimensional case, the stable/unstable manifolds were first constructed by \cite{EW_CSF}. In higher dimensions, the construction of asymptotic orbits using contraction mapping theory was established in \cite{CM_dynamics}, and the stable/unstable manifolds using the invariant cone technique from dynamical systems were established in \cite{SX1}.

The first-order term in the rescaled flow contains geometric information about the singularity beyond its tangent flow. For cylindrical singularities, this phenomenon appears in the work of the third-named author, Zhihan Wang, and Jinxin Xue \cite{SWX1,SWX2}, where the leading asymptotic profile is related to the local structure of the singular set. In particular, the first-order term can distinguish isolated cylindrical singularities from non-isolated ones and, in the latter case, determine curvature information for the singular set. A related phenomenon occurs for flows modeled on Simons cone: the work of Angenent--Daskalopoulos--Sesum \cite{ADS26}, building on the solutions of Vel\'azquez \cite{V94}, shows that the side through which a flow continues past the singularity is encoded in the first-order asymptotics.

We expect that the existence of prescribed first-order asymptotics of mean curvature flow near a singularity could be the first step towards the higher-order asymptotics. In the special situation of spherical singularities and cylindrical singularities, the third-named author, Zhihan Wang, and Jinxin Xue have used the first-order asymptotics to derive the next-order asymptotics near a cylindrical singularity, see \cite{SWX2}.

\subsection{Proof ideas}
\label{sec:idea}

The main theorems are based on the Wa\.{z}ewski box argument. 
It is a topological argument that has been used to construct singularities possibly with the first-order behavior prescribed for geometric flows in \cite{AV97, Sto, LZ, SS26}.

We briefly describe how Theorem \ref{thm:forced-realization} is proven for an asymptotically conical shrinker $\Sigma.$
The main ideas follow similar structures as in \cite{Sto, LZ}.
However, several new ingredients are necessary to deal with more complicated nature of the flow in the setting of Theorem~\ref{thm:forced-realization}.
First, we construct a closed embedded hypersurface $\Sigma^{\rm cap}_R$ that coincides with $\Sigma$ in a large ball $B_R$ outside which the curvature is very small.
The goal is to suitably perturb $\Sigma^{\rm cap}_R$ so that the mean curvature flow starting from it develops a singularity with the prescribed tangent flow.

To this end, we look at the space of flows starting from perturbations using truncated eigenmodes.
These flows are assumed to have suitable decaying rates on both their stable and unstable components and form a finite-dimensional box.
A large portion of the arguments is dedicated to establishing an improved interior estimate to show that such a perturbed flow can leave the box only from the unstable side.
A topological argument 
implies the existence of a desired solution with the prescribed singularity modeled on $\Sigma.$ 
Roughly speaking, if such a long-time solution did not exist, each parameter $\p$ in a small $m$-dimensional closed ball $B$ would lead to a solution $h_{\p}$ existing up to a finite time $\tau(\p).$
The fact that a perturbed flow can leave the box only from the unstable side allows us to construct a continuous map from $B$ to its boundary, an $(m-1)$-dimensional sphere.
Showing that this continuous map is a retraction leads to a contradiction.
Theorem~\ref{thm:first-order} follows from the same principle but with a more refined box in the topological argument.

We organize the paper as follows.
In Section~\ref{s:Pre}, we start with preliminary results and provide the evolution equations of the forced mean curvature flow.
In Section~\ref{sec: bar func}, we construct barrier functions that will be used in Section~\ref{sec:int-est} to prove the global $C^2$ improvement theorem.
In Section~\ref{s:box-forced}, we prove the main theorems about prescribing tangent flows.
In Section~\ref{s:prescribe-first-order}, we prove the theorem about prescribing first-order asymptotics.
In Section~\ref{sec:closed-case}, we indicate how the proofs can be simplified to deal with the case of closed shrinkers.
We construct a barrier function for $\n h$ in Appendix~\ref{s:barrier-calc}, provide detailed proofs for Propositions \ref{p:riem-admissible} and \ref{p:rot-admissible} in Appendix~\ref{app:geometric-forcing}, and construct the closed embedded hypersurface $\Sigma^{\rm cap}_R$ mentioned above in Appendix~\ref{sec:Cap}.

\subsection*{Acknowledgment}
J.C. is supported by the AMS-Simons Travel Grant.
T.-K.L. acknowledges the support from Man-Chun Lee and NSF Grant DMS-2529637.
A.S. is supported by the AMS-Simons Travel Grant and the Simons Foundation Travel Support for Mathematicians; he thanks Zhihan Wang for some helpful discussions.
X.Z. is supported by NSF Grant DMS 1812142, NSF Grant DMS 1811267 and NSF Grant DMS 2104349; he thanks Lu Wang for some helpful discussions.

\section{Preliminary}
\label{s:Pre}

\subsection{Shrinkers and linearization} \label{prel: shrinker and linear}
Let $\Sigma$ be an $n$-dimensional smooth embedded hypersurface in $\bb R^{n+1}.$
We say $\Sigma$ is a shrinker if its mean curvature $\vec H = \h$ satisfies \eqref{shrinker-eq}.
Shrinkers arise naturally when singularities of a mean curvature flow are studied via a blow-up analysis, which has another interpretation by the rescaled mean curvature flow, as we will explain in the following.

Let $M$ be an $n$-dimensional manifold and $F\colon M\times[-1,0)\to\bb R^{n+1}$ be an embedded mean curvature flow in the sense that $\bd_t F = \h,$ the mean curvature vector of $F(\cdot,t).$
The rescaled flow
\begin{align*}
\td F(x,\tau)
:= e^{\tau/2}\cdot 
F\pr{x, -e^{-\tau}}
\end{align*}
satisfies the rescaled mean curvature flow equation 
$\bd_\tau \td F = \td\h + \td F/2.$
At $(0,0)\in\RR^{n+1}\times\RR$, a tangent flow of $F$ is modeled on a shrinker $\Sigma$ if and only if the rescaled flow $\td F(\cdot,t)$ subsequentially converges to the embedding of $\Sigma$ in the locally smooth sense.
Thus, to study the behavior of solutions to the rescaled flow, it is important to understand its linearization $(\bd_t-L)u=0$ where 
\begin{align*}
L
=L_\Sigma
:= \D_\Sigma 
- \frac 12 \pair{x, \n_\Sigma(\cdot)}
+ |A_\Sigma|^2 
+ \frac 12.
\end{align*}
The operator $L$ is self-adjoint with respect to the Gaussian weighted $L^2$-inner product $\pair{u,v}_{L^2_W}
:= \int_\Sigma u(x)\, v(x)\,e^{-|x|^2/4}d\cH^n(x);$ see \cite{CM12}. 
The corresponding $L^2$ and $H^k$ spaces will be denoted by $L^2_W=L^2_W(\Sigma)$
and $H^k_W=H^k_W(\Sigma).$ 

\subsection{Asymptotically conical shrinkers}
In the rest of the paper, we will mostly work with an asymptotically conical shrinker.
In Section~\ref{sec:closed-case}, we will indicate how the same proofs apply to a closed shrinker after suitable simplifications.

A smooth embedded shrinker $\Sigma\sbst\bb R^{n+1}$ is called asymptotically conical if its blowdown is a regular cone; that is, $\sqrt{-t}\Sigma\to\mathcal C$ in $C^\infty_{\rm loc}(\bb R^{n+1}\setminus\{0\})$ as $t\to 0^-$ where $\mathcal C$ is a cone whose link is a smooth embedded hypersurface in $S^{n}.$
As explained in \cite{BW17, CS, LZ}, there exists an $L^2_W$-orthonormal basis $\{h_i\}_{i\in\bb N}$ of $L^2_W(\Sigma)$ such that 
\begin{equation}
\label{e:L-eigenfunction}
Lh_i=-\lambda_i h_i.
\end{equation}
The eigenvalues $\lambda_i$'s satisfy $\lambda_1<\lambda_2\le\lambda_3\le\cdots$ tending to $\infty$ and each eigenspace $E_{\lambda_i}$ is finite-dimensional.
Each of the eigenfunction $h_i$ has polynomial growth; see \cite{CM25} and \cite{LZ}*{Theorem 2.4}.

Some weighted spaces considered in \cite{CS} will be used when we establish the estimates.
Let $\tilde{r}\geq 1$ be a smooth function such that $\tilde{r}=|x|+1$ for $|x|\geq1$. 
We consider the weighted norms 
\begin{align*}
\|u\|_{0;-\gamma}^{\rm hom}&:=\sup_{x\in\Sigma}\tilde{r}(x)^{\gamma}|u(x)|,\text{ and}\\
[u]_{\alpha;-\gamma-\alpha}^{\rm hom}&:=\sup_{x,y\in \Sigma}\frac{1}{\tilde{r}(x)^{-\gamma-\alpha}+\tilde{r}(y)^{-\gamma-\alpha}}\frac{|u(x)-u(y)|}{|x-y|^\alpha}.
\end{align*}
The space $C^{0,\alpha}_{{\rm hom},-\gamma}$ is defined to collect all $u$ such that
\begin{align*}
\|u\|_{C^{0,\alpha}_{{\rm hom},-\gamma}}
:=\|u\|_{0;-\gamma}^{\rm hom}+[u]^{\rm hom}_{\alpha;-\gamma-\alpha}<\infty,
\end{align*}
and the space $C^{1,\alpha}_{{\rm hom},-\gamma}$ is defined to collect all $u$ such that
\begin{align*}
\|u\|_{C^{1,\alpha}_{{\rm hom},-\gamma}}
:=\sum_{j=0}^1\|(\na_\Sigma)^ju\|_{C^{0,\alpha}_{{\rm hom},-j-\gamma}}<\infty.
\end{align*}
We will mostly use the norms when $\gamma=-1,$ that is, the $C^{k,\alpha}_{{\rm hom},1}$-norms.

In the rest of the paper, we will mostly view hypersurfaces and flows by their embedding maps into the Euclidean space.
For example, on a shrinker $F\colon\Sigma\to\RR^{n+1},$ the shrinker equation \eqref{shrinker-eq} and
\begin{align}\label{shrinker-div}
{\rm div}_{\frac{|F|^2}{4}}(V)
= {\rm div}(V)-\ppair{V,\,\frac{F}{2}},
\end{align}
imply
\begin{align}\label{shrinker-F-Laplacian}
\Delta \frac{|F|^2}{4}-\na_{\na^g \frac{|F|^2}{4}}\frac{|F|^2}{4}=\frac{n}{2}-f,
\end{align}
where we let $f:=|F|^2/4$ and let $g$ be the metric induced by the Euclidean one.


\subsection{Forced mean curvature flow}
\label{s:fmcf}

In this section, we consider the forced mean curvature flow
\begin{equation}
\label{e:fmcf}\partial_tF_0=\h(F_0,t)+\phi_0[F_0,t]\cdot\N_{F_0},
\end{equation}
for a family of immersions $F_0 : M\times I \to \mathbb{R}^{n+1}$ where $M$ tentatively denotes an abstract manifold.
The main theorems will be proven after they are reduced to two different types of forced mean curvature flows in Section \ref{s:rot-riem}.
After rescaling, the equation for the rescaled forced mean curvature flow $F_r\colon M\times [0,\infty)\to\bb R^{n+1}$ is
\begin{equation} \label{e:rfmcf}
\partial_\tau F_r=\h_{F_r}+\frac{1}{2}F_r+{\phi_r}[F_r,\tau]\cdot \N_{F_r}.    
\end{equation}

We fix an asymptotically conical shrinker $\Sigma^n \subset \mathbb{R}^{n+1}$ and view it as an embedding $F: \Sigma \to \mathbb{R}^{n+1}.$ 
If $F_r$ can be written as the graph of a function $h(x,t)$ over $\Sigma$, then as the calculation in \cite{LZ}*{Appendix A}, we denote \begin{align*}
g_0(t)(X,Y)=\la X(F_0(\cdot,t)),\,Y(F_0(\cdot,t))\ra,
\end{align*}
and construct a family of diffeomorphisms $\varphi_t$ of $\Sigma$ and functions $h$ such that \begin{align}\label{e:F_0}
\frac{1}{\sqrt{-t}}F_0(\varphi_t(x),t)
= F(x) + h(x,t)\N_F(x)
= F_r(x,t),
\end{align}
which implies that
\begin{align*}
&(\partial_tF_0)(\varphi_t(x),t)+dF_0(\varphi_t(x),t)(\partial_t\varphi_t(x))+\frac{1}{2\sqrt{-t}} F(x)
=\pr{\sqrt{-t}\,\partial_t{h}(x,t)-\frac{1}{2\sqrt{-t}}{h}(x,t)}\N_F(x).
\end{align*}
Thus, for $\tau=-\ln(-t)$, $\partial_t = \frac{1}{-t} \partial_{\tau}$ and 
\begin{align}\label{h-equ}
\frac{\partial h}{\partial \tau}
=& Lh+\mathcal{Q}_1+\mathcal{Q}_2,
\end{align}
where
\begin{align*}
\mathcal{Q}_1
=&{\sqrt{-t}}\ppair{(\partial_tF_0-\Delta^{g_0}F_0)(\varphi_t(x),t),\N_{F_0}(\varphi_t(x))} \frac{1}{{\la\N_F(x),\N_{F_r}(x)\ra}}, 
\text{ and}\\ 
\mathcal{Q}_2=&C_3(F,\na F,\na^2 F,h)*h+C_4(F,\na F,\na^2 F,h,\na h)*\na h+C_5(F,\na F,\na^2 F,h,\na h)*\na^2 h \notag.
\end{align*}
Compared with the discussion in the Euclidean case, the key difference is that previously in \cite{LZ}, $(\partial_tF_0-\Delta^{g_0}F_0)(\varphi_t(x),t)=0$ on $\set{(x,\tau)\in \Sigma\times[0,\infty):f(x)<\gamma_2 e^{(\tau+\tau_0)}}$ for a large constant $\gamma_2$
and nonzero outside because of the cutoff. 
However, here it is nonzero in general due to the forcing term. 

The main aim of this paper is to prove Theorem~\ref{thm:forced-realization}, where the flow $F_0(x,t)$ with additional forces satisfies the following assumption.

\begin{Assumption}[Assumption on additional forces]
\label{t:rmcf}
Let \(f=|F|^2/4\) and $\td f=f+1,$
and let $\tilde r\geq1$ be smooth with
$\tilde r(x)=|F(x)|+1$ for $|F(x)|\geq1$.
Fix $\Gamma_0\geq4$.  
We assume that there are
$\delta_0,\ep_0>0$ sufficiently small and $\tau_0>0$ sufficiently
large such that the following holds.
\begin{enumerate}
\item 
If
\(
    \tilde r^{-1}|h|+|\na h|+|\na^2h|
    \leq\ep<\ep_0,
\)
then
\begin{equation}\label{e:forced nonlinear}
\begin{aligned}
\mathcal Q_1
&=
\frac{\sqrt{-t}
\ppair{(\partial_tF_0-\Delta^{g_0}F_0)(\varphi_t(x),t),
\N_{F_0}(\varphi_t(x),t)}}
{\la\N_F(x),\N_{F_r}(x,\tau)\ra}                                      \\
&=
\frac{\phi_r[F(x)+h(x,\tau)\N_F(x),\tau]}
{\la\N_F(x),\N_{F_r}(x,\tau)\ra}                                      \\
&=
Y_0(x,\tau)
+Y_1(h,\tilde r\na h,x,\tau)h
+Y_2^m(h,\tilde r\na h,x,\tau)(\na h)_m
+Y_3^{ij}(h,\tilde r\na h,x,\tau)(\na^2h)_{ij}.
\end{aligned}
\end{equation}
Here $Y_j=Y_j(a,p,x,s)$, where $a$ and $p$ denote the fibre
variables corresponding to $h$ and $\tilde r\na h$, respectively;
$\na_x$ differentiates in $x$ with $(a,p,s)$ fixed, and
$\na_{a,p}^{\beta}$ denotes a fibre derivative of total order $\beta$.

\item 
For all nonnegative integers $l,k,\beta$ with $2l+k+\beta\leq4n+6$, whenever
\(
    |F(x)| \leq \Gamma_0e^{\frac12(\tau+\tau_0)},
\)
we have
\begin{align}
|\partial_s^l\na_x^kY_0|(x,\tau)
&\leq
\delta_0e^{-\frac12(\tau+\tau_0)}\tilde f^{-\frac{k}{2}},
\label{e:admissible-Y0}\\
\tilde r^\beta
|\partial_s^l\na_x^k\na_{a,p}^{\beta}Y_1|
(h,\tilde r\na h,x,\tau)
&\leq
\delta_0e^{-(\tau+\tau_0)}\tilde f^{-\frac{k}{2}},
\label{e:admissible-Y1}\\
\tilde r^\beta
|\partial_s^l\na_x^k\na_{a,p}^{\beta}Y_2^m|
(h,\tilde r\na h,x,\tau)
&\leq
\delta_0e^{-\frac12(\tau+\tau_0)}\tilde f^{-\frac{k}{2}},
\label{e:admissible-Y2}\\
\tilde r^\beta
|\partial_s^l\na_x^k\na_{a,p}^{\beta}Y_3^{ij}|
(h,\tilde r\na h,x,\tau)
&\leq
C_{l,k,\beta}\eta_{\rm amb}e^{-(\tau+\tau_0)}
\tilde f^{1-\frac{k}{2}}
\leq
\delta_0\tilde f^{-\frac{k}{2}}.
\label{e:admissible-Y3}
\end{align}
The constants are chosen so that
$\la\N_F,\N_{F_r}\ra\geq\frac12$ and the full second-order
coefficients remain uniformly elliptic.

\item
We also assume short-time existence, uniqueness, continuous dependence,
and the standard continuation criterion for $F_0$.
For every capped perturbation used below, let $M_t$, $-1\leq t<T\leq0$,
be its maximal smooth flow, and define
\(
    f_M = |F_M|^2/4
\)
on the initial capped hypersurface.  
The exterior pseudolocality
requirement is
\begin{equation}\label{e:exterior-regularity}
|\na^jA_{M_t}|(F_0(x,t))
\leq
C_j\Gamma_0^{-\frac{j+1}{2}}
\,\,\text{ on }
\set{(x,t)\,\,|\,\, f_M(x)\geq\frac12\Gamma_0,\, -1\leq t<T}
\end{equation}
for $0\leq j\leq4n+6,$ where $C_j$ is independent of $\tau_0$ and of the permitted perturbation.  
This estimate is obtained from buffered initial geometric bounds on
$\{f_M\geq\frac14\Gamma_0\}$.
\end{enumerate}

For a closed shrinker, the exterior requirement
\eqref{e:exterior-regularity} is omitted.
\end{Assumption}
This assumption is required to ensure that the additional forces would not beat the rescaling. 
If the additional forces are too large, then the blow-up limit may not even be the shrinker, and the system would behave essentially differently.

\subsection{Rotationally symmetric and Riemannian mean curvature flows}
\label{s:rot-riem}

We verify Assumption~\ref{t:rmcf} for the Riemannian and rotationally symmetric flows.
The detailed proofs are in Appendix~\ref{app:geometric-forcing} and we provide the main proof ideas here.

Suppose that the rescaled flow is written
as
\(
    F_r=F+h\N_F
\)
over the shrinker \(F\colon\Sigma\to\RR^{n+1}\), and assume
\(\tilde r^{-1}|h|+|\na h|+|\na^2h|
    \leq\ep<\ep_0.\)
Since \(\tilde f^{1/2}\simeq\tilde r\), after reducing \(\ep_0\),
\[
    \la\N_F,\N_{F_r}\ra\geq\frac12.
\]
All \(x\)-derivatives of \(Y_j(a,p,x,s)\) are taken with \((a,p,s)\) fixed and then evaluated at \((a,p,s) = (h,\tilde r\na h,\tau)\).

\begin{prop}
\label{p:riem-admissible}
Let \((N^{n+1},g)\) be a smooth Riemannian manifold and \(p\in N\).  
For every fixed \(\Gamma_0\), after choosing the normal-coordinate homothety sufficiently small and \(\tau_0\) sufficiently large, the rescaled Riemannian mean curvature flow starting from the capped perturbation constructed in Appendix~\ref{sec:Cap} satisfies Assumption~\ref{t:rmcf}.
\end{prop}

\begin{proof}[Proof sketch]
Let \(\bar g_\tau\) be the rescaled ambient metric.  For
\(2l+|\alpha|\leq4n+8\), the normal-coordinate homothety gives
\begin{align*}
 \big|\partial_\tau^l\partial^\alpha
       (\bar g_\tau-\delta)\big|
 &\leq
 C_{l,\alpha}\eta_{\rm amb}e^{-(\tau+\tau_0)}
 (1+|\xi|)^{2-|\alpha|},\\
 \big|\partial_\tau^l\partial^\alpha
       \Gamma_{\bar g_\tau}\big|
 &\leq
 C_{l,\alpha}\eta_{\rm amb}e^{-(\tau+\tau_0)}
 (1+|\xi|)^{1-|\alpha|},
\end{align*}
where \(\eta_{\rm amb}\to0\) with the homothety scale.  The difference
between the Riemannian and Euclidean mean-curvature operators is
affine-linear in \(\na^2h\).  The calculations in
Appendix~\ref{app:riemannian-forcing} therefore give
\eqref{e:forced nonlinear} and
\eqref{e:admissible-Y0}--\eqref{e:admissible-Y3} for
\(2l+k+\beta\leq4n+6\), provided
\(
    C_{n,\Sigma}\eta_{\rm amb}\Gamma_0^2\leq\delta_0.
\)

For the exterior estimate, extend the metric from a normal-coordinate
ball to a complete bounded-geometry metric, without changing it near
the capped flow.  A strictly convex normal-coordinate function and the
maximum principle keep the flow inside the unchanged region.
Chen--Yin pseudolocality \cite{CY} and the interior derivative estimates then
give \eqref{e:exterior-regularity}.  Standard compact mean curvature
flow theory gives the remaining well-posedness statements.
\end{proof}

\begin{prop}
\label{p:rot-admissible}
Let \(M\) be any capped quotient profile, including the permitted perturbation constructed in Appendix~\ref{sec:Cap}.   
Translate the rotation axis to
\(
    y=-C_{\mathrm{ax}}e^{\tau_0/2},
\)
and let \(M'\subset\RR^n\times\RR^m\) be its \(SO(m)\)-invariant lift.
If \(C_{\mathrm{ax}}=C_{\mathrm{ax}}
(\delta_0,\Gamma_0,m,\Sigma)\) is sufficiently large, then the rescaled quotient flow starting from $M$ satisfies Assumption~\ref{t:rmcf}.
\end{prop}

\begin{proof}[Proof sketch]
After translating the axis and rescaling, the rotational forcing is
\[
    -(m-1)e^{-\frac12(\tau+\tau_0)}
    \frac{(\partial_y)^\perp}
    {C_{\mathrm{ax}}
     +e^{-\frac12(\tau+\tau_0)}y}.
\]
The cap construction and the choice of \(C_{\mathrm{ax}}\) ensure that the denominator is bounded below by \(C_{\mathrm{ax}}/2\), both on the normal-graph region and on the entire initial capped profile.  
Moreover,
\[
\frac{\la\partial_y,\N_{F+h\N_F}\ra}
     {\la\N_F,\N_{F+h\N_F}\ra}
=
\la\N_F,\partial_y\ra
-
\la\partial_iF,\partial_y\ra
\big((I+hA_F)^{-1}\big)^{ij}(\na h)_j.
\]
Hence \eqref{e:forced nonlinear} holds with \(Y_3^{ij}=0\), and the
estimates \eqref{e:admissible-Y0}--\eqref{e:admissible-Y2} follow by
differentiating the explicit coefficients; see
Appendix~\ref{app:rotational-forcing}.

The full lift \(M'_t\) is an ordinary Euclidean mean curvature flow.
Comparison with shrinking cylinders preserves its distance from the rotation axis.  
The strengthened cap estimates, Euclidean pseudolocality, and the interior derivative estimates therefore give
\eqref{e:exterior-regularity}.  
Standard compact mean curvature flow theory gives the remaining well-posedness statements.
\end{proof}

Thus, Propositions~\ref{p:riem-admissible} and \ref{p:rot-admissible} verify Assumption~\ref{t:rmcf} for the Riemannian and rotationally symmetric flows.  
Consequently, Theorems~\ref{thm:riemannian-realization} and \ref{thm:rotational-realization} follow from Theorem~\ref{thm:forced-realization}.

\section{Barrier functions} \label{sec: bar func}

In this section, we construct barrier functions that will be used in the proof of the global $C^2$ improvement theorem.
Similar to~\cite{Sto}*{Section~6.2}, our barrier function will be of the form
\begin{equation} \label{barrier func}
 u= e^{-2\lambda_*(\tau+\tau_0)} \pr{D\tilde{f}^{2\lambda_*+1}-B\tilde{f}^{2\lambda_*}}
\end{equation}
for some fixed $\lambda_*\in(0,1/2)$ and some positive constants $B$ and $D$, where 
\begin{equation*}
    \tilde{f}(x)=f(x)+1=\frac{|F(x)|^2}{4}+1.
\end{equation*}
By choosing the constants $B$ and $D$ suitably, the function $u$ satisfies $(\partial_\tau-Q)u\geq 0$ on the corresponding region for a comparison operator $Q$ defined in \eqref{oper}.
With these choices, Proposition~\ref{barrier} gives the desired estimate.
Unlike in the Ricci flow case, we need a similar construction for $\nabla h$, as explained at the end of this section.
The calculations are similar and are included in Appendix~\ref{s:barrier-calc}.
Together, the barriers for $h$ and $\na h$ provide the large-scale input for the global $C^2$ improvement at the end of Section~\ref{sec:int-est}.

We start from the evolution inequality for the square of the normal graph function over an asymptotically conical shrinker $F\colon \Sigma\to\RR^{n+1}$.  
Using the graph equation and the notation of~\cite{LZ}, one obtains after \eqref{h-equ} that
\begin{equation}
    \label{e:h2equa}
\begin{split}
    &\quad \partial_\tau h^2
    =2h\bd_\tau h
    = 2h(Lh+\mathcal{Q}_1+\mathcal{Q}_2)\\ 
    &\leq \pr{\Delta^g-\na^g_{\na^g \frac{|F|^2}{4}}} |h|^2 - 2|\na^g h|^2
    + \pr{2|A_g|^2+1}h^2
    + C^{ab}\na_a^g\na_b^gh^2-2C^{ab}\na_a^gh\na^g_bh\\  
    &\quad +2|h\mathcal{Q}_1|
    +|h|\cdot |C_3*h+C_4*\na^g h|. 
\end{split}
\end{equation}
Here, the part of $\mathcal{Q}_2$ that contains second derivatives of $h$ is written in the form
\begin{equation}
\label{1storder}
\begin{split}
C^{ab}\na^g_a\na^g_bh
:=& |(I+hA_g)^{-1}(\na^gh)|^2(\Delta^gh)\\
&+(1+|(I+hA_g)^{-1}(\na^gh)|^2)(\tilde{h}^{ab}-|\na h|_{g}^2g^{ab})\na^g_a\na^g_bh.
\end{split}
\end{equation}
For the remaining lower-order terms, we use $C=C(F,\na F,\na^2F,h,\na h)$ to denote a tensor whose coefficients depend only on the listed quantities.  With this convention,
\begin{align*}
     |C_3*h|\leq& h^2\cdot|\na A_g|\cdot |A_g|,\text{ and}\\
    |C_4*\na^g h|\leq&  C_n|\na^gh|^2(1-|hA_g|)^{-2}((|hA_g|(1+|\na^gh|)+1+|\na^g h|^2)\cdot|A_g|+|\na A_g|\cdot|h|\cdot(|hA_g|+|\na^g h|))\\&+C_n(|\na A_g|\cdot|h|\cdot|\na^g h|\cdot(1+(1-|hA_g|)^{-1})+|A_g|\cdot(|\na^g h|^2+|hA_g|\cdot|\na^g h|)).
\end{align*}
Combining \eqref{1storder} with Cauchy's inequality gives the following bounds.  
If $\|h\|_{C^{1,\alpha}_{\rm hom,1}}\leq \ep<\ep_{n,\Sigma}$,  then
\begin{align}
    |C^{ab}|\leq C_{n,\Sigma}\pr{|hA|+|\na^g h|^2}\leq& C_{n,\Sigma}\ep,\label{e:Cab}\\
    |h|\cdot |C_3*h+C_4*\na^g h|\leq& {C_{n,\Sigma}}(\ep|\na^gh|^2+(|A|^2+\ep|\na A|)h^2),\label{e:C3C4}\\ 
    |h\mathcal{Q}_1-hY_3^{ij}(\na^2 h)_{ij} |\leq &\delta_0|h|e^{-\frac{1}{2}(\tau+\tau_0)}(1+e^{-\frac{1}{2}(\tau+\tau_0)}|h|+|\na h|),\label{e:Q_1-Y_3}\\
    hY_3^{ij}(\na^2 h)_{ij}=&\frac{1}{2}Y^{ij}_3\na_i^g\na^g_jh^2-Y_3^{ij}\na^g_ih\na^g_jh,\\
    |Y_3^{ij}|\leq& C_0\eta e^{-(\tau+\tau_0)}f\leq \delta_0,
\end{align}
where \eqref{e:Q_1-Y_3} follows from \eqref{e:admissible-Y0}--\eqref{e:admissible-Y2}.
These five inequalities are used to obtain the upper bound~\eqref{e:equofh2} for the derivative of $h^2$.

On the region $\{f(x)\leq \Gamma e^{\tau+\tau_0}\}$, we choose $\delta_0$ small enough so that $|Y_3^{ij}|\leq \delta_0<\ep_0$.  Taking $\ep\leq \ep_{n,\Sigma}$ then makes the gradient contribution strictly favorable: $$-2|\na^gh|^2-2(C^{ab}+Y_3^{ab})\na_a^gh\na_b^gh+(C_{n,\Sigma}\ep+\delta_0)|\na^gh|^2<0.$$
Consequently,
\begin{equation}
    \label{e:equofh2}
\begin{split}
     \partial_\tau|h|^2
     \leq &\pr{g^{ab}+C^{ab}+Y_3^{ab}} \na_a^g\na_b^gh^2
     -\na^g_{\na^g \frac{|F|^2}{4}}|h|^2 \\
     &+\pr{C_n|A|^2+C_n|\na A|+1+3\delta_0e^{-\tau-\tau_0}}h^2+2\delta_0e^{-\frac{1}{2}(\tau+\tau_0)}|h|.
\end{split}
\end{equation}
This motivates the comparison operator
\begin{equation}
    \label{oper}
\begin{split}
    Q\psi
    :=&\pr{g^{ab}+C^{ab}+Y_3^{ab}}\na_a^g\na_b^g\psi
    -\na^g_{\na^g \frac{|F|^2}{4}}\psi\\
    &+ \pr{C_n|A|^2+C_n|\na A|+1+3\delta_0e^{-\tau-\tau_0}}\psi
    +2\delta_0e^{-\frac{1}{2}(\tau+\tau_0)}\sqrt{|\psi|}.
\end{split}
\end{equation}
Recall that we let $\tilde{f}=\frac{|F|^2}{4}+1$, which is positive on $\Sigma$.  
Using the relation \eqref{shrinker-F-Laplacian}, for any exponent~$k$, we obtain
\begin{align} \label{e:Qfk}
    g^{ab} \na_a^g\na_b^g\tilde{f}^k
    -\na^g_{\na^g \frac{|F|^2}{4}}\tilde{f}^k
    = k\tilde{f}^{k-1}\pr{\frac{n}{2}+1-\tilde{f}} 
    + k(k-1)\tilde{f}^{k-2}{|\na^g \tilde{f}|^2}.
\end{align}
Combined with that $\na^g_a\na^g_b\tilde{f}=\na^g_a\na^g_b\frac{|F|^2}{4}=\frac{\la F_a,F_b\ra}{2}-\frac{\la F,A_{ab}\N \ra}{2}$, the perturbative second-order term is estimated by 
\begin{align}\label{Cest}
    \abs{(C^{ab}+Y_3^{ab})\na_a\na_b\tf^k}
    & = \abs{k\tf^{k-1}(C^{ab}+Y_3^{ab})\na_a^g\na^g_b\tf+k(k-1)(C^{ab}+Y_3^{ab})\tf^{k-2}\na_a^g\tf\na^g_b\tf} \\
    & \leq \abs{k\tf^{k-1}} \cdot
    \abs{C^{ab}+Y_3^{ab}}
    \pr{|A_{ab}|\tf^{\frac{1}{2}}+(\frac{1}{2}+|k-
    1|)|F_a||F_b|}.\nonumber
\end{align}

We next build the large-scale barrier.  
Following the structure of the Ricci-flow barrier in~\cite{Sto}*{Section~6.2}, fix $0<\lambda_*<\frac{1}{2}$ and consider the function defined by \eqref{barrier func} on the region
\[
\set{(x,\tau)\in \Sigma\times[0,\infty):\Gamma <f(x)<\Gamma e^{\tau+\tau_0}}.
\]
Using \eqref{e:Qfk} and \eqref{Cest} together with ${|\na f|^2}\leq f$, we compute 
\begin{align*}
   & e^{2\lambda_*(\tau+\tau_0)}(\partial_\tau-Q)u\\
   =& e^{2\lambda_*(\tau+\tau_0)}(\partial_\tau-Q) \pr{e^{-2\lambda_*(\tau+\tau_0)}\pr{D\tf^{2\lambda_*+1}-B\tf^{2\lambda_*}}}\\ 
   =&-(\frac{n}{2}+1) \pr{(2\lambda_*+1)D\tf^{2\lambda_*}-2\lambda_*B\tf^{2\lambda_*-1}} +B\tf^{2\lambda_*} -(({2\lambda_*+1})(2\lambda_*) D\tf^{2\lambda_*} - (2\lambda_*)(2\lambda_*-1)B\tf^{2\lambda_*-1}) \frac{|\na \tf|^2}{\tf}\\  
   &- (C^{ab}+Y_3^{ab}) \na_a\na_b \pr{D\tf^{2\lambda_*+1}-B\tf^{2\lambda_*}} -\pr{C_n|A|^2+C_n|\na A|+3\delta_0e^{-\tau-\tau_0}} \pr{D\tf^{2\lambda_*+1}-B\tf^{2\lambda_*}} \\  
   &-e^{2\lambda_*(\tau+\tau_0)} \sqrt{|u|}2\delta_0e^{-\frac{1}{2}(\tau+\tau_0)}\\ 
   \ge& \tf^{2\lambda_*}\left(B + \frac{2\lambda_*(n+2)B}{2\tf} -\frac{(n+2)({2\lambda_*+1})D}{2} - {(2\lambda_*+1)}(2\lambda_*)D\frac{|\na \tf|^2}{\tf}+2\lambda_*(2\lambda_*-1)\frac{B}{\tf}\frac{|\na \tf|^2}{\tf} \right.\\ 
   &-D({2\lambda_*+1}) \abs{C^{ab}+Y_3^{ab}} \abs{A_{ab}} \tf^{\frac{1}{2}}- D{(2\lambda_*+1)}({2\lambda_*+\frac{1}{2}})\abs{C^{ab}+Y_3^{ab}}-B(2\lambda_*) \abs{C^{ab}+Y_3^{ab}} \abs{A_{ab}} \tf^{-\frac{1}{2}} \\ &- \frac{B(2\lambda_*)({\frac{3}{2}-2\lambda_*})\abs{C^{ab}+Y_3^{ab}}}{\tf}- \pr{C_{n,\Sigma}|A|^2+C_{n,\Sigma}|\na A|+3\delta_0e^{-\tau-\tau_0}} (D\tf-B))\Bigg) \\ &-e^{2\lambda_*(\tau+\tau_0)}\sqrt{|e^{-2\lambda_*(\tau+\tau_0)}(D\tf^{2\lambda_*+1}-B\tf^{2\lambda_*})|}2\delta_0e^{-\frac{1}{2}(\tau+\tau_0)}.
\end{align*}
The last term is harmless on the large-scale region.  Indeed, if $D>\frac{2B}{\Gamma}$
\begin{align}
    &e^{2\lambda_*(\tau+\tau_0)}\sqrt{|e^{-2\lambda_*(\tau+\tau_0)}(D\tf^{2\lambda_*+1}-B\tf^{2\lambda_*})|}2\delta_0e^{-\frac{1}{2}(\tau+\tau_0)}\\ 
    \leq& 2\delta_0e^{(\lambda_*-\frac{1}{2})(\tau+\tau_0)}\sqrt{D}\tf^{\lambda_*+\frac{1}{2}}
    =2\delta_0\sqrt{D}{\left(\frac{\tf}{e^{\tau+\tau_0}}\right)}^{\frac{1}{2}-\lambda_*}\tf^{2\lambda_*}
    \leq 2\delta_0\sqrt{D}(2\Gamma)^{\frac{1}{2}-\lambda_*}\tf^{2\lambda_*}.\notag
\end{align}
Fix \(T>0\) and denote
\[
\mathcal A_{T,\Gamma,\tau_0}
:=
\left\{
(x,\tau)\in\Sigma\times[0,T]:
\Gamma <f(x)<
\Gamma e^{\tau+\tau_0}
\right\}.
\]
Here again \(f=|F|^2/4\) and \(\widetilde f=f+1\).  
On this region,
\[
\widetilde f\geq\Gamma +1
\,\,\text{ and }\,\,
\frac{\widetilde f}{e^{\tau+\tau_0}}
\leq\Gamma+e^{-\tau_0}.
\]

Fix once and for all a constant \(C_{n,\Sigma}\geq1\)
dominating all constants in the preceding estimates, and set
\[
\beta:=
4C_{n,\Sigma}
\left[
(2\lambda_*+1)
\left(2\lambda_*+1+\frac n2\right)
+\sup_\Sigma
\left(
\widetilde f|A|^2+
\widetilde f|\nabla A|
\right)+1
\right].
\]
Choose \(\Gamma\) and \(\tau_0\) so that
\[
\Gamma+1
\geq
\max\{2\beta,16C_{n,\Sigma}\}.
\]
Assume
\(
\|h\|_{C^{1,\alpha}_{\rm hom,1}}\leq\varepsilon
\)
and assume \(\varepsilon,\delta_0>0\) are sufficiently small so that
\[
C_{n,\Sigma}(\varepsilon+\delta_0)
\left(
1+\frac{\beta}{\Gamma +1}
\right)
\leq\frac{\beta}{16}
\,\,\text{ and }\,\,
3\delta_0\left(\Gamma+e^{-\tau_0}\right)
\leq\frac{\beta}{16}.
\]
Define $B$ and $D$ by
\begin{equation}\label{BD-choices}
\begin{aligned}
B&:=\beta D,\\
D&:=
\max\left\{
\sup_{\partial_p\mathcal A_{T,\Gamma,\tau_0}}
\frac{
e^{2\lambda_*(\tau+\tau_0)}h^2
}{
\widetilde f^{2\lambda_*}
(\widetilde f-\beta)
},
\left[
\frac{32\delta_0}{\beta}
\left(\Gamma+e^{-\tau_0}\right)^{
\frac12-\lambda_*}
\right]^2
\right\},
\end{aligned}
\end{equation}
where \(\partial_p\mathcal A_{T,\Gamma,\tau_0}\) denotes the
parabolic boundary of the moving region. Then we have the following barrier estimate.

\begin{Pro}[Large-scale barrier; cf. the Ricci-flow estimates in \cite{Sto}*{Lemmas~6.5--6.7}]\label{barrier}
    Assume $\|h\|_{C^1_{\rm hom,1}}\leq \ep<\ep_n$.  Then $h^2$ is a subsolution of the operator $Q$ in the sense that
    $(\partial_\tau-Q)h^2\leq 0$, where $Q$ is defined in \eqref{oper}.
    In addition, assume that $\mathcal Q_1$ satisfies Assumption~\ref{t:rmcf}.
For the constants $C_{n,\Sigma}$, $\beta$, and $\Gamma$ chosen above, and
for $B$ and $D$ given in \eqref{BD-choices}, on $\set{(x,\tau)\in \Sigma\times[0,\infty):\Gamma <f(x)<\Gamma e^{\tau+\tau_0}},$
    the function $u$ defined in~\eqref{barrier func} is a supersolution; that is, $(\partial_\tau-Q)u\geq 0$ on this region. Consequently,
\[
h^2\leq u
\qquad\text{on }
\mathcal A_{T,\Gamma,\tau_0}.
\]
\end{Pro}

\begin{proof} By our choice of the constants $\Gamma$, $B$, and $D$, the function
\[
u:=
e^{-2\lambda_*(\tau+\tau_0)}
\left(
D\widetilde f^{2\lambda_*+1}
-B\widetilde f^{2\lambda_*}
\right)
=
D e^{-2\lambda_*(\tau+\tau_0)}
\widetilde f^{2\lambda_*}
(\widetilde f-\beta)
\]
is positive on \(\mathcal A_{T,\Gamma,\tau_0}\), satisfies
\(h^2\leq u\)
on \(\partial_p\mathcal A_{T,\Gamma,\tau_0},\)
and obeys
\[
(\partial_\tau-Q)u
\geq
\frac{B}{4}
e^{-2\lambda_*(\tau+\tau_0)}
\widetilde f^{2\lambda_*}>0
\]
throughout \(\mathcal A_{T,\Gamma,\tau_0}\).
Hence the comparison principle gives the desired estimate.
\end{proof}

Proposition~\ref{barrier} controls the homogeneous $C^0$ size of $h$.  
In the mean curvature flow setting, however, this is not sufficient by itself, as the normal graph function may have linear growth at infinity, so the rescaling argument used in the Ricci-flow case does not directly upgrade $C^0$ control to the required gradient control.  
We therefore need to repeat the barrier construction for $\nabla h$.
The calculations are similar, but we include them in Appendix~\ref{s:barrier-calc} for completeness.

\section{Interior estimates}
\label{sec:int-est}

We next record the interior estimates needed for the forced mean curvature flow equation.  
The ultimate goal of this section is to derive a global $C^2$-improvement estimate, which is presented in Theorem \ref{impro}.
It is a consequence of interior estimates that we will start to recall and establish.
In the proof of the main theorem, the improvement theorem will be crucial when we argue that the flow can leave the designed ``box'' only through the unstable side; see Section~\ref{s:box-forced}.

The starting point is the following scaled H\"older estimate, where we defined the scaled H\"older norm as
\begin{align}
    \|u\|_{C^{2m,\alpha}(\Omega)}:=\sum_{i+2j\leq 2m}r_\Omega^{i+2j}(\|\na^i\partial^j_tu\|_{C^0(\Omega)}+r_\Omega^\alpha[\na^i\partial^j_tu]_{\alpha,\frac{\alpha}{2}}).
\end{align}

\begin{Pro}[Scaled H\"older estimate; cf. \cite{LZ}*{Proposition~4.3}, and for Ricci-flow analogues see \cite{Bam,App,Sto}]
\label{GHolder}
Let $2r\geq s>0$, and set $\Omega=B_r\times [0,s^2]$ and $\Omega'=B_{2r}\times [0,s^2]$.  
Suppose that $u\in C^2(\Omega')$ satisfies
\begin{equation}
	\label{e:schauder}
\begin{split}
	&(\partial_t-L)u=Q[u]+F\\
	=&r^{-2}f_1(r^{-1}x,r^{-2}t,u,r\nabla u)\cdot u+r^{-1}f_2(r^{-1}x,r^{-2}t,u,r\nabla u)\cdot \nabla u\\ 
	&+ f_3(r^{-1}x,r^{-2}t,u,r\nabla u)\cdot \nabla u\otimes\nabla u+f_4(r^{-1}x,r^{-2}t,u,r\nabla u)\cdot u\otimes\nabla^2 u\\ 
	&+f_5(r^{-1}x,r^{-2}t,u,r\nabla u)\cdot r\nabla u\otimes\nabla^2 u+F(x,t)
\end{split}
\end{equation}
with initial data $u(\cdot,0)=u_0$, where $f_1,f_2,f_3,f_4$, and $f_5$ are smooth.  
The operator $L$ has the form
    \begin{align*}
        Lu =a^{ij}(x,t)\partial^2_{ij}u+b^i(x,t)\partial_iu+c(x,t)u,
    \end{align*}
and for $m\geq 1$, its coefficients obey, 
\begin{align*}
    \frac{1}{\Lambda}\leq a^{ij}\leq \Lambda,\,
    \|a^{ij}\|_{C^{2m-2,\alpha}(\Omega')}\leq \Lambda,\,
    \|b^{i}\|_{C^{2m-2,\alpha}(\Omega')}\leq r^{-1}\Lambda,\,\text{ and }
    \|c\|_{C^{2m-2,\alpha}(\Omega')}\leq r^{-2}\Lambda.
\end{align*}
Then there exist constants $\ep>0$ and $C<\infty$, depending only on $n,m,\alpha,\Lambda$, and the functions $f_i$, such that if
\(\|u\|_{C^0(\Omega')}+\|u_0\|_{C^{2m,\alpha}(B_{2r})}+r^2\|F\|_{C^{2m-2,\alpha}(\Omega')}
\leq \ep,\)
then
\begin{align*}
    \|u\|_{C^{2m,\alpha}(\Omega)}
    \leq C\pr{\|u\|_{C^0(\Omega')}+\|u_0\|_{C^{2m,\alpha}(B_{2r})}+r^2\|F\|_{C^{2m-2,\alpha}(\Omega')}}.
\end{align*}
\end{Pro}

We now explain why this estimate applies to the forced mean curvature flow equation \eqref{e:forced nonlinear}.  
The only term requiring comment is the coefficient of $\na^2h$.  
By the fundamental theorem of calculus,
\begin{equation}
	\label{e:FTC}
\begin{split}
	& Y_{3}^{ij}(h,\tilde{r}\,\na h,x,\tau)(\na^2h)_{ij}\\
	=&\,\pr{Y_{3}^{ij}(0,0,x,\tau)
	+\hat{Y}_{3}^{ij}(h,\tilde{r}\na h,x,\tau)\cdot h
	+\tilde{Y}_{3}^{ij}(h,\tilde{r}\na h,x,\tau)\cdot\tilde{r}\na h}
	(\na^2h)_{ij}.
\end{split}
\end{equation}
The first term can be moved to the left-hand side of \eqref{e:schauder}, whereas the remaining two terms have the nonlinear structure allowed on the right-hand side.  
All other terms are handled similarly, except for $Y_0(x,\tau)$, which is treated as the inhomogeneous forcing term $F$.  
We use the decay estimate \eqref{e:admissible-Y0} and these give the following local H\"older estimate.

\begin{Pro}[Forced MCF H\"older estimate; cf. \cite{LZ}*{Proposition~4.4}, and for Ricci-flow analogues see \cite{Sto}*{Lemma~6.2}] \label{hHolder}
Let $\Sigma$ be an asymptotically conical shrinker, and let $h$ solve \eqref{h-equ} on $\Omega'\times (\tau_1,\tau_3)$, where $\Omega\Subset \Omega'\Subset \Sigma$.  
Fix $m\in\mathbb{N}$ and times $\tau_1\leq \tau_2< \tau_3$.  Assume also that $\mathcal Q_1$ satisfies Assumption~\ref{t:rmcf}.  Then there are constants
$\ep = \ep(n,m)>0$ and
$C = C\pr{n,m,\Sigma,F,\dist_\Sigma(\Omega,\Sigma\setminus\Omega'),\tau_2-\tau_1} < \infty$
such that the following estimates hold.\\
(1) If $\tau_2>\tau_1$, then 
\begin{align*}
    \|h\|_{L^\infty(\Omega'\times[\tau_1,\tau_3])}<\ep\Rightarrow \|h\|_{C^m(\Omega\times[\tau_2,\tau_3])}\leq C(\|h\|_{L^\infty(\Omega'\times[\tau_1,\tau_3])}+\delta_0e^{-\frac{1}{2}(\tau_1+\tau_0)}).
\end{align*}
(2) If $\tau_2=\tau_1$, then
\begin{align*}
    &\|h\|_{L^\infty (\Omega'\times[\tau_1,\tau_3])} + \|h\|_{C^{m+1}(\Omega'\times\{\tau_1\})}<\ep\\
    \Rightarrow& \|h\|_{C^m(\Omega\times[\tau_1,\tau_3])}
    \leq C\pr{\|h\|_{L^\infty(\Omega'\times[\tau_1,\tau_3])}+\|h\|_{C^{m+1}(\Omega'\times\{\tau_1\})}+\delta_0e^{-\frac{1}{2}(\tau_1+\tau_0)}}.
\end{align*}
\end{Pro}

We also need the corresponding local energy estimate.

\begin{Lem}[Local integral estimate; cf. \cite{LZ}*{Lemma~4.5}, and for Ricci-flow analogues see \cite{Sto}*{Lemma~6.3}]
    \label{hIntegral}
Let $\Sigma$ be an asymptotically conical shrinker, and let $h$ solve \eqref{h-equ} on $\Omega'\times (\tau_1,\tau_3)$, where $\Omega\Subset \Omega'\Subset \Sigma$.  Fix $m\in\mathbb{N}$ and times $\tau_1\leq \tau_2< \tau_3$.  Assume that $\mathcal Q_1$ satisfies Assumption~\ref{t:rmcf}.  
Then there exist constants
$\ep = \ep(n,m)>0$ and
$C = C(n,m,\Sigma,F,\dist_\Sigma(\Omega,\Sigma\setminus\Omega'),\tau_2-\tau_1)<\infty$
with the following properties.\\
(1) If $\tau_2>\tau_1$ and $\|h\|_{C^{m+2}(\Omega'\times[\tau_1,\tau_3])}<\ep,$ then
\begin{align*}
	\sup_{\tau\in[\tau_2,\tau_3]} \|\na^mh(\tau)\|_{L^2_W(\Omega)} + \|h\|_{H_W^{m+1}L^2(\Omega\times[\tau_2,\tau_3])}
	& \leq C\pr{\|h\|_{H_W^mL^2(\Omega'\times[\tau_1,\tau_3])}+\delta_0e^{-\frac{1}{2}(\tau_1+\tau_0)}}.
\end{align*}
(2) If $\tau_2=\tau_1$ and $\|h\|_{C^{m+2}(\Omega'\times[\tau_1,\tau_3])}<\ep,$ then
\begin{align*}
	\sup_{\tau\in[\tau_1,\tau_3]} \|\na^mh(\tau)\|_{L^2_W(\Omega)}+\|h\|_{H_W^{m+1}L^2(\Omega\times[\tau_1,\tau_3])}
	&\leq C\pr{\|h(\tau_1)\|_{H_W^m(\Omega')}+\|h\|_{H_W^mL^2(\Omega'\times[\tau_1,\tau_3])}+\delta_0e^{-\frac{1}{2}(\tau_1+\tau_0)}}.
\end{align*}
\end{Lem}

\begin{proof}
The proof is a standard local energy argument.  
We indicate only the changes caused by the forcing terms.  
Compared with the estimates in \cite{LZ}*{Lemma~4.5} and \cite{Sto}*{Lemma~6.3}, the new contribution is
$Y_0(x,\tau)+Y_1(h,\tilde{r}\,\na h,x,\tau)h+Y_2^m(h,\tilde{r}\,\na h,x,\tau)(\na h)_m+Y_{3}^{ij}(h,\tilde{r}\,\na h,x,\tau)(\na^2h)_{ij}$.
We write $N(x,\tau)=Y_0(x,\tau)$.
By \eqref{shrinker-div} and \eqref{e:FTC}, the equation can be arranged as
\begin{align*}
    \partial_\tau h
    = \dL h +\sum_{l=0}^1B_l^0*\na^lh-{\rm div}_{\frac{|F|^2}{4}}\pr{\hat{h}^0*\na h}+N,
\end{align*}
where $\hat{h}^{0}$ is a function of $h$ and $\n h,$ 
$\dL f:=\Delta^gf-\frac{1}{2}\la F,\na^gf\ra$ is the drift Laplacian, and
${\rm div}_{\frac{|F|^2}{4}}(T)={\rm div}(T)-\frac{1}{2}\la F,T\ra.$
The coefficients $B_l^0$ depend on the derivatives of $h$ up to order two and are small under the hypotheses.  
The symbol $*$ denotes a fixed linear combination of the indicated tensor contractions.  

Given $C_0>0$ and $\delta>0,$ there exists $\ep>0$ such that
\begin{align*}
    \|h(\tau)\|_{C^2(\Omega')}<\ep\Rightarrow \|B_1^0\|_{L^\infty(\Omega')}+\|B_0^0\|_{L^\infty(\Omega')}<C_0
    \,\text{ and }\,\|\hat{h}^0\|_{L^\infty(\Omega')}\leq \delta.
\end{align*}
Commuting derivatives and arguing by induction gives
\begin{align*}
    \partial_\tau\na^mh
    = \dL \na^mh +\sum_{l=1}^{m+1}B_l^m*\na^lh-{\rm div}_{\frac{|F|^2}{4}}\pr{\hat{h}^m*\na^{m+1} h}+\na^mN,
\end{align*}
where, for some $C_m$ and for every $\delta>0$, one can choose $\ep>0$ so that
\begin{align*}
    \|h(\tau)\|_{C^{m+2}(\Omega')}<\ep\Rightarrow \sum_{l=0}^{m+1}\|B_l^m\|_{L^\infty(\Omega')}<C_m
    \,\text{ and } \,
    \|\hat{h}^m\|_{L^\infty(\Omega')}\leq \delta.
\end{align*}
The estimate then follows by integration by parts, Young's inequality, and the estimate
\begin{align*}
    \|Y_0\|_{H_W^mL^2(\Omega'\times[\tau_1,\tau_3])}\leq \delta_0e^{-\frac{1}{2}(\tau_0+\tau_1)}
\end{align*}
that follows from the decay \eqref{e:admissible-Y0} of the forcing term.
\end{proof}

The previous estimates imply the compact-region interior estimate that will be used later.  
The restriction $\lambda_*<\frac{1}{2}$ is essential in this step.

\begin{Pro}[Compact-region interior estimate; cf. the Ricci-flow estimate \cite{Sto}*{Lemma~6.4}]\label{int1}
    Let $\Gamma>0$, and let $h$ solve \eqref{h-equ} on $\Sigma\times (0,\tau_1)$.  
    Assume that $\mathcal Q_1$ satisfies Assumption~\ref{t:rmcf} and that
    \begin{align*}
        \|h\|_{C^{0,\alpha}_{\rm hom,1}(\Sigma\times [0,\tau_1])}&<\ep,\\
        \|h(\tau)\|_{L^2_W(\Sigma)}&\leq \mu e^{-\lambda_*(\tau+\tau_0)}\,\,\text{ for all}\,\,\tau\in[0,\tau_1], \text{ and}\\
        \sum_{l=0}^{4n+3} \abs{\na^lh}_g(x,0)&\leq C_0\bar{p}e^{-\lambda_*\tau_0}\,\,\text{ for all}\,\,x\in \set{\frac{|F|^2}{4}<2\Gamma}\text{ and }\tau=0
    \end{align*}
    for some $0<\lambda_*<\frac{1}{2}$, where $\lambda_*$ does not belong to the spectrum of $L$.
    If $\Gamma>\Gamma_0(n,\Sigma,F)$, $\ep<\ep_0(n,\Sigma,F,\Gamma)$, and $\tau_0>C(n,\mu,\Sigma,F,\lambda_*,\Gamma,C_0\bar{p})$, then there exists $C_1=C_1(n,\Sigma,F,\lambda_*,\Gamma)<\infty$ such that
    \begin{align*}
        \sum_{l=0}^2|\na^l h|_{g}(x,\tau)\leq C_1(\mu+C_0\bar{p})e^{-\lambda_*(\tau+\tau_0)}
        \,\,\text{ for all}\,\,
        (x,\tau)\in \set{\frac{|F|^2}{4}<\Gamma }\times [0,\tau_1].
    \end{align*}
\end{Pro}

The same local estimate also gives large-scale interior bounds after a rescaling adapted to the shrinker.  
Recall from~\cite{CS} the family of diffeomorphisms $\varphi_t:\Sigma\to \Sigma$ generated by $X_t=\frac{1}{-2t}x^T$ with $t=-e^{-\tau}$; namely
\(\bd_t\varphi_t = X_t\circ\varphi_t\)
with $\varphi_{-1}=\text{Id}.$
Set $\hat{g}_t=(-t)\varphi_t^*g_\Sigma$.  
This family converges to the metric cone asymptotic to $\Sigma$ in the pointed $C^\infty$ Cheeger--Gromov sense.  
Define
\begin{align*}
    \mathcal{H}(x,t)=\sqrt{-t}\cdot h(\varphi_t(x)\,,\,-\log(-t)).
\end{align*}
Then \eqref{h-equ} implies
\begin{align}\label{H-equ}
    \frac{\partial  \mathcal{H}}{\partial t}(x,t)
    &=\left({\Delta^{\hat{g}_t}  \mathcal{H}}+\frac{ \mathcal{H} |A(\varphi_{t}(x))|^2}{-t}\right)(x,t)
    +\frac{1}{\sqrt{-t}}(\mathcal{Q}_1+\mathcal{Q}_2)(\varphi_{t}(x)\,,\,-\log(-t)).
\end{align}
The convergence of $\hat g_t$ to the cone gives uniform comparison of the corresponding homogeneous norms:
\begin{align}
    C^{-1}\|{h}\|_{{\mathcal{C}_{\text{hom},1}^{*,\alpha}}(B_{{2R}/{\sqrt{-t}}}\setminus B_{{R}/{\sqrt{-t}}})}
    \leq\| \mathcal{H}\|_{{\mathcal{C}_{\text{hom},1}^{*,\alpha}}(B_{2R}\setminus B_R)}
    \leq  C\|{h}\|_{{\mathcal{C}_{\text{hom},1}^{*,\alpha}}(B_{{2R}/{\sqrt{-t}}}\setminus B_{{R}/{\sqrt{-t}}})}.
\end{align}
Thus, $\frac{1}{\sqrt{-t}}\mathcal{Q}_2(\varphi_{t}(x)\,,\,-\log(-t))$ can be expressed in terms of $\mathcal H$, $\na\mathcal H$, and $\na^2\mathcal H$.  
We also use the following comparison between $\varphi_t$ and the pure conical dilation.  
If the end of $\Sigma$ is parametrized by $F:(R,\infty)\times \Gamma\to \Sigma$, let $\tilde{\varphi}_t=F^{-1}\circ \varphi_t\circ F$ and let $\phi_t:(R,\infty)\times \Gamma\to (R,\infty)\times \Gamma$ be given by
$(r,\omega)\mapsto ((-t)^{-\frac{1}{2}}r,\omega)$.

\begin{Lem}[\cite{CS}*{Lemma~3.8}]\label{comp}
For $j\geq 1$ and $\eta>0$,
    \begin{align*}
d_{g_{\mathcal{C}}}(\tilde{\varphi}_t(r,\theta),\phi_t(r,\theta))\lesssim\frac{1}{\sqrt{-t}\cdot r}
\,\,\text{ and }\,\,
|D^{(j)}(\tilde{\varphi}_t-\phi_t)|(r,\theta)\lesssim \frac{1}{\sqrt{-t}\cdot r^{1+j-\eta}}.
\end{align*}
\end{Lem}

Applying $\nabla^{\hat g_t}$ to \eqref{H-equ} gives
\begin{equation}
	\label{naH-equ}
\begin{split}
	\frac{\partial}{\partial t} (\nabla^{\hat{g}_t}\mathcal{H})
	=&\Delta^{\hat{g}_t}(\nabla^{\hat{g}_t}\mathcal{H})
	+\left(\frac{|A(\varphi_t(x))|^2}{-t}-\text{Ric}_{\hat{g}_t}\right) \nabla^{\hat{g}_t}\mathcal{H}\\
	& +\mathcal{H}\nabla^{\hat{g}_t}\left(\frac{|A(\varphi_t(x))|^2}{-t}\right)
	+\nabla^{g_\Sigma}(\mathcal{Q}_1+\mathcal{Q}_2)(\varphi_{t}(x)\,,\,-\log(-t)).
\end{split}
\end{equation}
Using \eqref{e:fmcf}, \eqref{e:forced nonlinear}, \eqref{e:FTC}, Lemma~\ref{comp}, and Proposition~\ref{GHolder} in \eqref{naH-equ}, the forcing term can be taken as
\begin{align*}
    F=&\mathcal{H}\left(\nabla^{\hat{g}_t}\left(\frac{|A(\varphi_t(x))|^2}{-t}\right)+\nabla^{g_\Sigma}C_3+\nabla^{g_\Sigma}Y_1(0,0,\cdot,\cdot)\right) +\nabla^{g_\Sigma}Y_0\\ 
    =&O\left(\|\mathcal{H}\|_{C^0_{\text{hom},1}}\right) \left(|x|^{-2}+\Gamma_0\delta_0e^{-(\tau+\tau_0)}\right)
    +\delta_0e^{-(\tau+\tau_0)}.
\end{align*}
Together with the pointed $C^\infty$ Cheeger--Gromov convergence of $\hat g_t$ to the cone, this yields the large-scale interior estimate.  
The argument is the mean-curvature-flow analogue of the rescaled estimates in \cite{Sto}*{Lemma~6.8}.

\begin{Lem}[Large-scale interior estimate; cf. \cite{LZ}*{Lemmas~4.14 and~4.15}, and for Ricci-flow analogues see \cite{Sto}*{Lemmas~6.8 and~6.9}]\label{int2}
      Let $h$ solve \eqref{h-equ}, and assume that $\mathcal Q_1$ satisfies Assumption~\ref{t:rmcf}.  
      Fix $0<\lambda_*<\frac{1}{2}$ with $\lambda_*$ outside the spectrum of $L$.  
      Suppose that
    \begin{align*}
         |h|_g&\leq C_0e^{-\lambda_*(\tau+\tau_0)} \pr{\frac{|F|^2}{4}(x)}^{\lambda_*+\frac{1}{2}}\text{on}
        \,\,\set{(x,\tau)\in\Sigma\times (0,\tau_1):\,\Gamma<\frac{|F|^2}{4}<\Gamma e^{\tau+\tau_0}}, \\ 
        |\na^gh|_g&\leq C_0e^{-\lambda_*(\tau+\tau_0)} \pr{\frac{|F|^2}{4}(x)}^{\lambda_*}\,\,\text{on}
        \,\,\set{(x,\tau)\in\Sigma\times (0,\tau_1):\,\Gamma <\frac{|F|^2}{4}<\Gamma e^{\tau+\tau_0}}, \text{ and}\\
        |\na^lh|_g(x,0)&\leq C_0e^{-\lambda_*\tau_0} \pr{\frac{|F|^2}{4}(x)}^{\lambda_*+\frac{1}{2}-\frac{l}{2}}\,\,\text{for all}\,\,
        l\in\{0,1,2,3\},\,\,x\in \set{\Gamma<\frac{|F|^2}{4}<\Gamma e^{\tau+\tau_0}}.
    \end{align*}
    Then, for $\Gamma>\Gamma_0(n,\Sigma,F)$, there exists $C_1=C_1(n,\Sigma,F,\lambda_*)<\infty$ such that
    \begin{align*}
        |\na^l h|_{g}(x,\tau)\leq C_1C_0e^{-\lambda_*(\tau+\tau_0)}\pr{\frac{|F|^2}{4}(x)}^{\lambda_*}\,\,\text{on}\,\,\set{(x,\tau)\in\Sigma\times (0,\tau_1):\,4\Gamma \leq\frac{|F|^2}{4}(x)\leq\frac{1}{4}\Gamma e^{\tau+\tau_0}}
    \end{align*}
    for every $l\in\{1,2\}$.
\end{Lem}

Combining the compact-region estimate in Proposition~\ref{int1} with the large-scale estimate in Lemma~\ref{int2} and the barrier functions in Section \ref{sec: bar func}, we obtain the global $C^2$ improvement theorem below.
The proof follows Stolarski's large-scale bootstrapping argument in \cite{Sto}*{Section 6}, while the analytic inputs are the forced mean curvature flow estimates recorded above.
The extra homogeneous $C^1$ assumption reflects the normal-graph formulation: 
since $h$ may grow linearly at infinity, $C^0$ controls are not sufficient.


\begin{Thm}[Quantitative global \(C^2\)-improvement;
cf. \cite{LZ}*{Theorem~5.10} and
\cite{Sto}*{Theorem~6.1}]
\label{impro}
Fix
\[
0<\lambda_*<\frac12,
\qquad
\lambda_*\notin\operatorname{spec}(L),
\]
and take the constant \(\Gamma_0\) in
Assumption~\ref{t:rmcf} sufficiently large.
Let \(h\) solve \eqref{h-equ} on
\(\Sigma\times[0,\tau_1]\), where \(0<\tau_1<\infty\).
Assume that the total term \(\mathcal Q_1\), including any transplantation error, admits the decomposition in the last line of
\eqref{e:forced nonlinear} and satisfies
\eqref{e:admissible-Y0}--\eqref{e:admissible-Y3}, with constants
independent of \(\tau_1\), and that the full second-order coefficient is
uniformly elliptic.  Suppose also that, for every
\(\tau\in[0,\tau_1]\),
\begin{align*}
\operatorname{supp}h(\cdot,\tau)
&\subset
\left\{f\leq\Gamma_0e^{\tau+\tau_0}\right\},
\\
\|h\|_{C^{1,\alpha}_{\rm hom,1}
(\Sigma\times[0,\tau_1])}
+\|\nabla^2h\|_{C^0(\Sigma\times[0,\tau_1])}
&\leq\ep,
\\
\|h(\tau)\|_{L_W^2}
&\leq
\mu e^{-\lambda_*(\tau+\tau_0)}.
\end{align*}
Define
\begin{align*}
K_{\rm in}
&:=
e^{\lambda_*\tau_0}
\max_{0\leq j\leq4n+4}
\sup_\Sigma
\tilde r^{-2\lambda_*-1+j}
|\nabla^jh(\cdot,0)|
\text{ and}\\
\Theta&:=\mu+K_{\rm in}+\delta_0.
\end{align*}
Then there exist constants
\[
\ep_*,\delta_*,\Theta_*>0,
\qquad
\tau_*<\infty,
\qquad
W<\infty
\]
depending on
\(n,\Sigma,F,\lambda_*,\Gamma_0\) and the fixed Hölder exponent, but
independent of \(h,\tau_0\), and \(\tau_1\), such that, whenever
\[
\ep\leq\ep_*,
\qquad
\delta_0\leq\delta_*,
\qquad
\Theta\leq\Theta_*,
\qquad
\tau_0\geq\tau_*,
\]
we have
\begin{align}
|\nabla^\ell h|(x,\tau)
&\leq
Wq(\Theta)\,
\tilde r(x)^{2\lambda_*+1-\ell}
e^{-\lambda_*(\tau+\tau_0)},
\qquad \ell=0,1,
\label{e:quantitative-C1-improvement}\\
|\nabla^2h|(x,\tau)
&\leq
Wq(\Theta)\,
\tilde r(x)^{2\lambda_*}
e^{-\lambda_*(\tau+\tau_0)}
\label{e:quantitative-C2-improvement}
\end{align}
for every \((x,\tau)\in\Sigma\times[0,\tau_1]\), where $\mathop{\lim}\limits_{\Theta\to0}q(\Theta)=0.$

In particular, for every prescribed
\(\varepsilon_{\rm pt}>0\), if the parameters are chosen so that
\(Wq(\Theta)\leq\varepsilon_{\rm pt}\), then the right-hand sides of
\eqref{e:quantitative-C1-improvement} and
\eqref{e:quantitative-C2-improvement} hold with
\(\varepsilon_{\rm pt}\) in place of \(Wq(\Theta)\).
\end{Thm}

\section{Box argument}
\label{s:box-forced}

In this section, we will prove the main theorem by adapting the Wa\.{z}ewski box argument to the forced mean curvature flow setting.  
Throughout this section, the forcing term is assumed to satisfy Assumption~\ref{t:rmcf}.  
In particular, after we write the rescaled flow as a normal graph over the shrinker, the forcing term contributes lower-order terms with the decay and smallness recorded in \eqref{e:admissible-Y0}--\eqref{e:admissible-Y3}.

Let $\Sigma$ be the concerned asymptotically conical shrinker and $F\colon\Sigma\to\RR^{n+1}$ be the corresponding embedding.
First, we choose $\tau_0$ and $\Gamma_0$ sufficiently large and follow Appendix~\ref{sec:Cap} to construct a closed hypersurface
\[
M=\Sigma^{\rm cap}_{4\sqrt{\Gamma_0} e^{\frac12\tau_0}}.
\]
We let $F_M\colon M\to\RR^{n+1}$ be the corresponding embedding map.

Fix $\lambda_*>0$ such that
\begin{align}\label{lambda-restrict}
	0<\lambda_*<\frac12
	\,\,\text{ and }\,\,
	\lambda_m<\lambda_*<\lambda_{m+1},
\end{align}
where $\lambda_m<\lambda_{m+1}$ are two different eigenvalues of $L=L_\Sigma.$
Recall that we let $h_i$ be the corresponding eigenfunctions.
We also fix a cutoff function $\eta:(0,\infty)\to[0,1]$ with $\eta=1$ on $(0,1/2)$ and $\eta=0$ on $(1,\infty)$.
For
\[
    \p=(p_1,\ldots,p_m)\in B_{\bar p e^{-\lambda_*\tau_0}}(0)\subset\mathbb R^m,
\]
we perturb $F_M$ by the unstable eigenfunctions, namely
\begin{align}\label{e:perturb-forced-box}
    F_{\p,0}(x)
    = F_M(x) + h_{\mathbf p}(x)\,\N_{F_M}(x),
\end{align}
where $\N_{F_M}$ denotes the chosen unit normal along $F_M$, $f(x)=|F_M(x)|^2/4,$ and
\[
h_{\mathbf p}(x)
:=
\eta\!\left(\frac{f(x)}{\gamma_0e^{\tau_0}}\right)
\sum_{i=1}^m p_i h_i(x).
\]  
The eigenfunction growth estimate gives
\begin{align}
\label{e:perturb-size-forced-box}
    \abs{\eta\pr{\frac{f(x)}{\gamma_0e^{\tau_0}}}
    \sum_{i=1}^m p_i h_i(x)}
    \leq C\bar p\gamma_0^{\lambda_*} |x|.
\end{align}
Thus, after taking $\bar p$ and $\gamma_0$ sufficiently small, the perturbed hypersurface $F_{\p,0}$ defined in \eqref{e:perturb-forced-box} remains embedded.  
We transplant it back to the original shrinker by
\begin{align*}
    \hat F_{\p}(x,t)
    =\eta\pr{\frac{f(x)}{\Gamma_0e^{\tau_0}}}F_{\p}(x,t)
    +\pr{1-\eta\pr{\frac{f(x)}{\Gamma_0e^{\tau_0}}}}F(x,t).
\end{align*}
As in the unforced case, for a positive time depending on $\p,$ we may write the rescaled transplanted flow as a normal graph over $\Sigma$ by
\begin{align*}
    \frac{1}{\sqrt{-t}}\hat F_{\p}(\varphi_t(x),t)
    =F(x)+h_{\p}(x,\tau)\N_F(x)
    \,\,\,\text{ for }\,\,
    \tau=-\log(-t).
\end{align*}
The support of $h_{\p}(\cdot,\tau)$ is contained in a compact subset of $\Sigma$ for each fixed $\tau$, since the transplanted flow agrees with the model shrinker outside the cutoff region.

The graph function satisfies an equation of the form
\begin{align}\label{h-equ-forced-box}
    \partial_\tau h_{\p}
    =Lh_{\p}+\mathcal Q_{1,\p}+\mathcal Q_{2,\p}.
\end{align}
Here, $\mathcal Q_{2,\p}$ denotes the nonlinear normal-graph error, including the usual quadratic terms coming from the graph parametrization.  
The term $\mathcal Q_{1,\p}$ contains both the cutoff error from the transplanted flow and the forcing terms.  
More explicitly, we write
\begin{equation}
	\label{e:cutoff-error-box}
\begin{split}
	\mathcal Q_{1,\p}
	:=& \sqrt{-t}\,
	\frac{
		\ppair{\bigl(\partial_t\hat F_{\p}-\Delta^{g_0}\hat F_{\p})(\varphi_t(x),t),\N_{F_1}(x)}}
	{\ppair{\N_{F_1}(x),\N_{F}(x)}}\\
	= & \, Y_0(x,\tau)
	+Y_1(h_{\p},\tilde r\nabla h_{\p},x,\tau)h_{\p}\\
	& +Y_2^a(h_{\p},\tilde r\nabla h_{\p},x,\tau)\nabla_a h_{\p}
	+Y_3^{ab}(h_{\p},\tilde r\nabla h_{\p},x,\tau)\nabla^2_{ab}h_{\p}.
\end{split}
\end{equation}
The denominator in \eqref{e:cutoff-error-box} is harmless as long as the graph remains in the box.  
The corresponding angle correction is included in $\mathcal Q_{2,\p}$.  
Assumption~\ref{t:rmcf} gives, for $\delta_0>0$ as small as needed,
\begin{align}\label{e:forced-smallness-box}
    \abs{Y_3^{ab}}&\leq\delta_0,\notag\\
    \abs{Y_0}+f^{\frac{1}{2}}\abs{\nabla Y_0}
    &\leq C\delta_0e^{-\frac12(\tau+\tau_0)},\\
    \abs{\mathcal Q_{1,\p}}
    &\leq C\delta_0e^{-\frac12(\tau+\tau_0)}
    \pr{1+e^{-\frac12(\tau+\tau_0)}\abs{h_{\p}}+\abs{\nabla h_{\p}}+\abs{\nabla^2h_{\p}}}.\notag
\end{align}

We start defining the box for which we will apply a topological argument as in \cite{Sto, LZ}.
For any $h\in L^2_W(\Sigma)$, define its unstable and stable projections by
\begin{align*}
    h_u
    :={}
    \sum_{j=1}^m
    \langle h,h_j\rangle_{L^2_W}h_j,
    \,\,\text{ and }\,\,
    h_s
    :={}
    \sum_{j=m+1}^{\infty}
    \langle h,h_j\rangle_{L^2_W}h_j
    =h-h_u.
\end{align*}
We define the box
\[
    \mathcal B[\lambda_*,\mu_u,\mu_s,\ep_0,\ep_1,\ep_2,t_1]
\]
to be the set of all functions
\[
    h:\Sigma\times[0,-\log(-t_1)]\to\mathbb R
\]
satisfying the following conditions.

\begin{enumerate}
    \item For every $\tau\in[0,-\log(-t_1)]$, the function
    $h(\cdot,\tau)$ is compactly supported in $\Sigma$.

    \item For every $\tau\in[0,-\log(-t_1)]$,
    \begin{align*}
        \|h_u(\cdot,\tau)\|_{L^2_W}
        \leq
        \mu_u e^{-\lambda_*(\tau+\tau_0)}, \qquad
        \|h_s(\cdot,\tau)\|_{L^2_W}
        \leq
        \mu_s e^{-\lambda_*(\tau+\tau_0)}.
    \end{align*}

    \item For every $\tau\in[0,-\log(-t_1)]$, the pointwise bounds
    \begin{align*}
        |h|_{\bar g}(x)
        \leq
        \ep_0\tilde r(x), \qquad
        |\nabla h|_{\bar g}(x)
        \leq
        \ep_1, \qquad
        |\nabla^2 h|_{\bar g}(x)
        \leq
        \ep_2
    \end{align*}
    hold on $\Sigma$.
\end{enumerate}
As before, $\tilde r$ is a fixed smooth function on $\Sigma$ satisfying $\tilde r\geq1$ and
\(\tilde r(x)=|x|+1\) for \(|x|\geq1.\)
Our goal is to show that, after choosing
\[
    \gamma_0,\bar p,\delta_0\ll1,
    \qquad
    \Gamma_0,\tau_0\gg1,
\]
the graph function $h_{\p}$ can leave the box only through the unstable side.

\begin{Pro}[Exit through the unstable side for the forced flow]\label{exit}
Fix $\lambda_*$ as in \eqref{lambda-restrict} and assume that the forcing term satisfies Assumption~\ref{t:rmcf}.  
After we choose $\delta_0,\gamma_0,\bar p,\varepsilon_0,\varepsilon_1,\varepsilon_2>0$ sufficiently small and $\Gamma_0,\tau_0<\infty$ sufficiently large, the following holds.

Let $h_{\p}$ be the graph function constructed above with
$|\p|\leq\bar p e^{-\lambda_*\tau_0}$ such that
\[
    h_{\p}\notin
    \mathcal B[\lambda_*,\mu_u,\mu_s,\ep_0,\ep_1,\ep_2,0].
\]
Define $t(\p)<0$ to be the largest time such that
\[
    h_{\p}\in
    \mathcal B[\lambda_*,\mu_u,\mu_s,\ep_0,\ep_1,\ep_2,t(\p)].
\]
Then, with $\tau(\p)=-\log(-t(\p))$, one has
\begin{align}\label{e:unstable-exit-forced}
    \|h_{\p,u}(\cdot,\tau(\p))\|_{L^2_W}
    =
    \mu_u e^{-\lambda_*(\tau(\p)+\tau_0)}.
\end{align}
Moreover, for every $\tau>\tau(\p)$ sufficiently close to $\tau(\p)$,
\begin{align}\label{e:hu1-forced}
    \|h_{\p,u}(\cdot,\tau)\|_{L^2_W}
    >
    \mu_u e^{-\lambda_*(\tau+\tau_0)}.
\end{align}
\end{Pro}

\begin{proof}
Short-time existence for the forced mean curvature flow and continuous dependence on the initial hypersurface imply that $t(\p)>-1$ when $\bar p$ and $\gamma_0$ are sufficiently small.

We first estimate the cutoff error.  
From \eqref{e:perturb-size-forced-box} and the curvature decay \eqref{cap-curv-decay} on the capping region, the initial hypersurface satisfies
\begin{align*}
    |\nabla^mA|(x,0)
    \leq
    D_m\pr{\frac{1}{\Gamma_0}}^{\frac{m+1}{2}}
\end{align*}
on $M\setminus\{f\leq \frac{\Gamma_0}2\}.$
The forcing term is perturbative under Assumption~\ref{t:rmcf}.  
Hence, the interior estimates and the pseudolocality estimates \cite{EH,INS}, applied with the forcing bounds, give
\begin{align*}
    |\nabla^mA|(x,t)
    \leq
    C_m\pr{\frac{1}{\Gamma_0}}^{\frac{m+1}{2}}
\end{align*}
for
$(x,t)\in \pr{M\setminus\{f\leq \Gamma_0/2\}}\times [-1,t(\p)].$ 
Consequently, while $h_{\p}$ stays in the box, the error $\mathcal E_{\p}$ contributed by cutoff in rescaled variables is controlled by 
\begin{align}\label{e:E-support-forced}
    \supp \mathcal E_{\p}(\cdot,\tau)
    &\subseteq
    \set{\frac{\Gamma_0e^{\tau+\tau_0}}{2}\leq f\leq\Gamma_0e^{\tau+\tau_0}},\notag\\
    |\mathcal E_{\p}(\cdot,\tau)|
    &\leq
    C\pr{\frac{1}{\Gamma_0e^{\tau+\tau_0}}}^{\frac12},\\
    |\nabla\mathcal E_{\p}(\cdot,\tau)|
    &\leq
    C\pr{\frac{1}{\Gamma_0e^{\tau+\tau_0}}}.\notag
\end{align}
These fit into Assumption \ref{t:rmcf} for $\tau_0$ sufficiently large.
Thus, the global $C^2$-improvement theorem, Theorem~\ref{impro}, applies to the forced equation \eqref{h-equ-forced-box}, and after we take $\tau_0$ large and the parameters~$\varepsilon_i$'s in the construction small, the pointwise part of the box is improved on $[0,\tau(\p)]$:
\begin{align}\label{e:C2-improve-forced-box}
    |h_{\p}|&\leq \frac12\ep_0\tilde r,
    &
    |\nabla h_{\p}|&\leq \frac12\ep_1,
    &
    |\nabla^2 h_{\p}|&\leq \frac12\ep_2.
\end{align}
In particular, $h_{\p}$ cannot exit through the $C^0$, $C^1$, or $C^2$ faces of the box.

We next record the $L^2_W$ size of the remaining errors.  
The nonlinear graph terms are quadratic under the improved estimates, and the angle correction in \eqref{e:cutoff-error-box} is also quadratic in the graph size.  
Therefore,
\begin{align}\label{e:nonlinear-error-forced}
    \|\mathcal Q_{2,\p}\|_{L^2_W}
    \leq
    C e^{-2\lambda_*(\tau+\tau_0)}.
\end{align}
For the forcing contribution, \eqref{e:forced-smallness-box} gives
\begin{align}\label{e:forcing-error-forced}
    \|\mathcal Q_{1,\p}-\mathcal E_{\p}\|_{L^2_W}
    \leq
    C\delta_0 e^{-\frac12(\tau+\tau_0)}.
\end{align}
Combining \eqref{e:C2-improve-forced-box}, \eqref{e:nonlinear-error-forced}, \eqref{e:forcing-error-forced}, 
and projecting \eqref{h-equ-forced-box} onto the eigenspaces of $L$, we obtain
\begin{equation}
	\label{e:unstable-ode-forced}
\begin{split}
	\frac{d}{d\tau}
	\|e^{\lambda_*(\tau+\tau_0)}h_{\p,u}\|_{L^2_W}
	&\geq
	(\lambda_* -\lambda_m-C\delta_0)
	\|e^{\lambda_*(\tau+\tau_0)}h_{\p,u}\|_{L^2_W}
	-\omega_{\tau_0,\delta_0}
	\,\,\text{ and}\\
	\frac{d}{d\tau}
	\|e^{\lambda_*(\tau+\tau_0)}h_{\p,s}\|_{L^2_W}
	&\leq
	(\lambda_* -\lambda_{m+1}+C\delta_0)
	\|e^{\lambda_*(\tau+\tau_0)}h_{\p,s}\|_{L^2_W}
	+\omega_{\tau_0,\delta_0},
\end{split}
\end{equation}
where the error is given by
\[\omega_{\tau_0,\delta_0}
:=C\pr{e^{-\lambda_*\tau_0}+\delta_0e^{-(\frac12-\lambda_*)\tau_0}}.\]
Since $0 < \lambda_*<1/2$, this error tends to zero as $\tau_0\to\infty$ once $\delta_0$ is fixed.

We now choose $\delta_0 > 0$ so small that
\begin{align*}
    \lambda_* -\lambda_m-C\delta_0>0
    \,\,\text{ and }\,\,
    \lambda_{m+1}-\lambda_*-C\delta_0>0.
\end{align*}
The initial data satisfy
\begin{equation}
	\label{e:initial-forced-box}
\begin{split}
	\abs{\la h_{\p}(\cdot,0),h_j\ra_{L^2_W}-p_j}
	&\leq
	Ce^{-\frac{\gamma_0}{100}e^{\tau_0}}
	\qquad\text{for }j\leq m,\\
	\|h_{\p,s}(\cdot,0)\|_{L^2_W}
	&\leq
	Ce^{-\frac{\gamma_0}{100}e^{\tau_0}}.
\end{split}
\end{equation}
Using the stable inequality in \eqref{e:unstable-ode-forced}, \eqref{e:initial-forced-box}, and the smallness of $\omega_{\tau_0,\delta_0}$, we obtain
\begin{align}\label{e:stable-strict-forced}
    \|h_{\p,s}(\cdot,\tau(\p))\|_{L^2_W}
    <
    \mu_s e^{-\lambda_*(\tau(\p)+\tau_0)}
\end{align}
for $\tau_0$ large enough.
Together with the improved pointwise bounds \eqref{e:C2-improve-forced-box}, this shows that the first exit cannot occur through the stable side or through the $C^2$ part of the box.  
Therefore, the only possible exit face is the unstable one, which proves \eqref{e:unstable-exit-forced}.

It remains to show that the exit is outward.  
At $\tau=\tau(\p)$, \eqref{e:unstable-exit-forced} gives
\[
    \|e^{\lambda_*(\tau(\p)+\tau_0)}h_{\p,u}(\cdot,\tau(\p))\|_{L^2_W}=\mu_u.
\]
Choosing $\tau_0$ larger, if necessary, so that
\[
    \omega_{\tau_0,\delta_0}
    <\frac12(\lambda_* -\lambda_m-C\delta_0)\mu_u,
\]
the first inequality in \eqref{e:unstable-ode-forced} implies that the rescaled unstable norm is increasing at $\tau(\p)$.  
Hence, \eqref{e:hu1-forced} holds for all $\tau>\tau(\p)$ sufficiently close to $\tau(\p)$.
\end{proof}

We next choose the perturbation parameters by a topological argument.
The proof is unchanged at the topological level; the only analytic input needed
for the forced flow is the strict outward crossing in Proposition~\ref{exit}.

\begin{Pro}[Choice of a perturbation remaining in the box]\label{perturbation}
Under the parameter choices in Proposition~\ref{exit}, and after taking
$\bar p<\mu_u/2$, there exists
\[
    \p\in \overline{B_{\bar p e^{-\lambda_*\tau_0}}(0)}\subset\mathbb R^m
\]
such that
\begin{align}\label{e:good-perturbation-forced}
    h_{\p}
    \in
    \mathcal B[\lambda_*,\mu_u,\mu_s,\ep_0,\ep_1,\ep_2,0].
\end{align}
\end{Pro}

\begin{proof}
Assume, toward a contradiction, that every
$\p\in \overline{B_{\bar p e^{-\lambda_*\tau_0}}}$ makes $h_{\p}$ leave the box before time $1$.  
For each such $\p$, let $\tau(\p)=-\log(-t(\p))$ be the first exit time
from Proposition~\ref{exit}.  
Define the rescaled exit map
\begin{align*}
    \mathcal F(\p)
    :=
    e^{\lambda_*(\tau(\p)+\tau_0)}
    \big(
    \langle h_{\p}(\cdot,\tau(\p)),h_1\rangle_{L^2_W},
    \ldots,
    \langle h_{\p}(\cdot,\tau(\p)),h_m\rangle_{L^2_W}
    \big)
    \in\RR^m.
\end{align*}
By Proposition~\ref{exit}, $|\mathcal F(\p)|=\mu_u$ for every $\p$ in the closed
parameter ball $\overline{B_{\bar p e^{-\lambda_*\tau_0}}}.$  
Moreover, the strict inequality \eqref{e:hu1-forced}, together
with the continuous dependence of the forced mean curvature flow on the initial hypersurface, implies that $\mathcal F$ is a continuous function on $\overline{B_{\bar p e^{-\lambda_*\tau_0}}}.$

On the boundary of the parameter ball, the initial projection estimates
\eqref{e:initial-forced-box} and the differential inequalities
\eqref{e:unstable-ode-forced} show that $\mathcal F$ has the same degree as the
radial map
\begin{align*}
    \p\longmapsto \mu_u\frac{\p}{|\p|}.
\end{align*}
Equivalently, after increasing $\tau_0$ and decreasing $\delta_0$ if necessary, the straight-line homotopy between these two boundary maps never hits the origin.
Thus, the restriction of $\mathcal F$ to
$\partial B_{\bar p e^{-\lambda_*\tau_0}}(0)$ has degree one as a map to the
sphere of radius $\mu_u$.
This is impossible because $\mathcal F$ is assumed to extend continuously to the
whole closed ball while taking values in the same sphere.  
Hence, the contradiction
assumption is false, so there exists a parameter $\p$ satisfying
\eqref{e:good-perturbation-forced}.
\end{proof}

\begin{proof}[Proof of Theorems~\ref{thm:forced-realization}, \ref{thm:riemannian-realization} and~\ref{thm:rotational-realization} ]
Proposition~\ref{perturbation} proves Theorem~\ref{thm:forced-realization} for an asymptotically conical shrinker.
As indicated in Section \ref{s:rot-riem}, this implies the main theorems~\ref{thm:riemannian-realization} and~\ref{thm:rotational-realization} for an asymptotically conical shrinker based on Propositions~\ref{p:riem-admissible} and \ref{p:rot-admissible}.    
\end{proof}

\section{Prescribing the first order asymptotics}
\label{s:prescribe-first-order}


In this section, we explain how to prescribe the leading order asymptotic of the forced rescaled mean curvature flow by refining the box argument in Section \ref{s:box-forced} and prove Theorem \ref{thm:first-order} for an asymptotically conical shrinker $\Sigma.$

Fix eigenvalues $\lambda_{m-1}<\lambda_{m}=\cdots=\lambda_M<\lambda_{M+1}=:\lambda_+$
of $L$ 
and assume
\begin{equation}
\label{e:lambda-first-order-range}
0<\lambda_{m}<\frac12.
\end{equation}
Set
\[
E_<
:=
\bigoplus_{\lambda_j<\lambda_{m}}E_{\lambda_j},
\quad
E_0
:=
\ker(L+\lambda_{m}),
\quad
E_>
:=
\overline{
	\bigoplus_{\lambda_j>\lambda_{M}}E_{\lambda_j}
}.
\]
We also write
\(E_\leq:=E_<\oplus E_0\),
and let \(\Pi_\leq:L^2_W\to E_\leq\)
and \(\Pi_>:L^2_W\to E_>\) be the corresponding orthogonal projections.

We use the forced graph equation in the form
\begin{equation}
\label{e:forced-graph-prescribe}
\begin{split}
\partial_\tau u
= &
Lu
+\mathcal E\\
& +Y_0
+Y_1(u,\tilde r\nabla u,x,\tau)u
+Y_2^a(u,\tilde r\nabla u,x,\tau)\nabla_a u\\
&
+Y_3^{ab}(u,\tilde r\nabla u,x,\tau)\nabla^2_{ab}u
+Q_2(u).
\end{split}
\end{equation}
Here, \(\mathcal E\) denotes the cutoff error from the compact gluing
construction, and \(Q_2(u)\) is the quadratic Euclidean graph
nonlinearity.  The construction gives
\[
\operatorname{supp}u(\cdot,\tau)
\subset\{f\leq\Gamma_0e^{\tau+\tau_0}\}.
\]
Extend the coefficients using a cutoff which equals one on this set
and is supported where
\(\tilde f\leq\frac14\Gamma_0^2e^{\tau+\tau_0}\); for
\(\Gamma_0\geq16\), this lies in the region of
Assumption~\ref{t:rmcf}.  Absorbing the transition errors into
\(\mathcal E\), we have, uniformly in the small graphical regime and
for \(2q+k\leq4n+6\),
\begin{align}
|\partial_\tau^q\nabla^k\mathcal E|
&\leq C_{q,k}\delta_0e^{-\frac12(\tau+\tau_0)}
\tilde f^{-k/2},
\label{e:E-pointwise-prescribe}\\
\|\mathcal E(\tau)\|_{L^2_W}
&\leq C_Ne^{-N(\tau+\tau_0)}
\quad\text{for every }N>0.
\label{e:E-Gaussian-prescribe}
\end{align}
We retain the notation \(Y_j\) for the extended coefficients.  If
\(\eta_0\) denotes the constant in \eqref{e:admissible-Y3} and
\(\nabla_{a,p}\) denotes fibre differentiation, then
\begin{equation}
\label{e:prescribe-Y3}
\tilde r^\ell
|\partial_\tau^q\nabla_x^k\nabla_{a,p}^\ell Y_3|
\leq C_{q,k,\ell}\eta_0e^{-(\tau+\tau_0)}
\tilde f^{1-\frac{k}{2}}
\leq\delta_0\tilde f^{-\frac{k}{2}},
\end{equation}
for \(2q+k+\ell\leq4n+6\), uniformly on the admissible fibre ball,
provided \(C\eta_0\Gamma_0^2\leq\delta_0\).

Define
\[
Y^{ab}_{3,0}(x,\tau):=Y^{ab}_3(0,0,x,\tau),
\qquad
P_\tau v:=Y^{ab}_{3,0}(x,\tau)\nabla^2_{ab}v,
\qquad
\mathcal D(P_\tau)=H^2_W(\Sigma).
\]
For \(V:=E_\leq\), the finite-dimensional eigenfunction estimates
give
\begin{equation}
\label{e:finite-spectral-estimates}
\begin{gathered}
|\nabla^jv|\leq C_V\|v\|_{L^2_W}\tilde r^{2\lambda_m+1-j},
\qquad
\|\tilde f^q\nabla^jv\|_{L^2_W}
\leq C_{V,j,q}\|v\|_{L^2_W},
\\
v\in V,\qquad 0\leq j\leq4n+4,\qquad q\geq0.
\end{gathered}
\end{equation}
See Section~\ref{prel: shrinker and linear} and
\cite[Theorem~2.4]{LZ}.
In particular, for every \(\varphi\in E_0\),
\begin{equation}
\label{e:Ptau-E0-automatic}
\|P_\tau\varphi\|_{L^2_W}
\leq C\eta_0e^{-(\tau+\tau_0)}\|\varphi\|_{L^2_W}.
\end{equation}

Choose
\begin{equation}
\label{e:beta-choice-prescribe}
\lambda_{m}
<
\beta
<
\min\left\{
\lambda_+,\,
2\lambda_{m},\,
\frac12
\right\}.
\end{equation}
Such a \(\beta\) exists by \eqref{e:lambda-first-order-range}.

\begin{thm}
	\label{t:prescribe-first-asymptotic}
	For every \(\psi\in E_0\), after choosing \(\delta_0,\gamma_0>0\)
	sufficiently small, \(\Gamma_0,\tau_0\) sufficiently large, and an
	additional force satisfying Assumption~\ref{t:rmcf} for these choices
	with \(C\eta_0\Gamma_0^2\leq\delta_0\),
	there exists a closed initial hypersurface such that the associated
	transplanted rescaled flow has a compactly supported normal-graph
	function \(u(\tau)\) over \(\Sigma\) satisfying
	\begin{equation}
	\label{e:prescribed-first-asymptotic}
	u(\tau)
	= e^{-\lambda_{m} (\tau+\tau_0)} \psi
	+ z(\tau),
	\end{equation}
	where
	\(\|z(\tau)\|_{L^2_W}
	\leq
	C e^{-\beta(\tau+\tau_0)}\)
	for all $\tau\geq 0.$
	Consequently,
	\begin{equation*}
	\lim_{\tau\to\infty}
	e^{\lambda_m(\tau+\tau_0)}u(\tau)
	=
	\psi
	\quad
	\text{in }L^2_W.
	\end{equation*}
\end{thm}

\begin{proof}
	Fix \(\psi\in E_0\) and set
	\(H_\psi(\tau)
	:=
	e^{-\lambda_m(\tau+\tau_0)}\psi.\)
	Since \(L\psi=-\lambda_m\psi\), we have
	\begin{equation}
	\label{e:Hpsi-linear}
	(\partial_\tau-L)H_\psi=0.
	\end{equation}
	We will construct \(u\) so that
	\(
	z:=u-H_\psi
	\)
	decays like \(e^{-\beta(\tau+\tau_0)}\).
	
	\medskip
	
	\noindent\emph{Step 1: Correcting the cutoff on the finite-dimensional space.}
	
	Choose
	\(\vartheta\in C^\infty([0,\infty),[0,1])\) with
	\(\vartheta=1\) on \([0,1]\) and \(\vartheta=0\) on
	\([2,\infty)\), and, for \(\sigma\geq\tau_0\), set
	\[
	\chi_\sigma(x)
	=
	\vartheta\left(\frac{\tilde f(x)}{\gamma_0e^\sigma}\right).
	\]
	Take \(0<\gamma_0<\Gamma_0/8\), so that
	\(\operatorname{supp}\chi_{\tau_0}\) lies in the part of the cap
	which agrees with \(\Sigma\).
	Define an operator \(T_{\tau_0}:V\to V\) by
	\[
	T_{\tau_0}v
	=
	\Pi_\leq(\chi_{\tau_0}v).
	\]
	Since \(V\) is finite-dimensional, its elements have polynomial growth.
	Therefore, the Gaussian tail estimate gives
	\[
	\|(1-\chi_{\tau_0})v\|_{L^2_W}
	\leq
	Ce^{-c\gamma_0e^{\tau_0}}\|v\|_{L^2_W}
	\]
	for all $v\in V.$
	Hence, we can estimate the operator norm
	\[
	\|T_{\tau_0}-I\|
	\leq
	Ce^{-c\gamma_0e^{\tau_0}},
	\]
	so for \(\tau_0\) sufficiently large, \(T_{\tau_0}\) is invertible with \(\|T_{\tau_0}^{-1}\|\leq2.\)
	
	Fix \(\mu>0\).  
	For
	\[
	a\in \overline B_\mu(V)
	:=
	\{a\in V:\|a\|_{L^2_W}\leq\mu\},
	\]
	define
	\begin{align*}
	r_\psi
	&:=
	\Pi_\leq
	\left(
	(\chi_{\tau_0}-1)e^{-\lambda_m\tau_0}\psi
	\right)\,
	\text{ and }\,
	w_a
	:=
	T_{\tau_0}^{-1}
	\left(
	e^{-\beta\tau_0}a-r_\psi
	\right).
	\end{align*}
	We take the initial graph to be
	\begin{equation}
	\label{e:initial-graph-prescribe}
	u_a(0)
	=
	\chi_{\tau_0}
	\left(
	e^{-\lambda_m\tau_0}\psi+w_a
	\right).
	\end{equation}
	As in the compact gluing construction, this compactly supported graph is inserted into the fixed exterior closed hypersurface, giving a closed initial hypersurface.
	
	Let
	\[
	z_a(0)
	:=
	u_a(0)-e^{-\lambda_m\tau_0}\psi.
	\]
	By the definition of \(w_a\),
	\begin{equation}
	\label{e:initial-low-projection}
	\Pi_\leq z_a(0)
	=
	e^{-\beta\tau_0}a.
	\end{equation}
	Moreover, for each fixed \(\gamma_0>0\), uniformly in
	\(a\in\overline B_\mu(V)\),
	\(\|\Pi_>z_a(0)\|_{L^2_W}
	=
	o(e^{-\beta\tau_0})\)
	as $\tau_0\to\infty,$
	since the high-frequency part only comes from cutting off finite-dimensional eigenfunction data hence is exponentially small in the Gaussian tail.  
	For \(\alpha\) chosen in Step~2, the same cutoff and eigenfunction
	estimates give
	\[
	K_{\gamma_0,\tau_0}:=
	e^{\alpha\tau_0}\max_{0\leq j\leq4n+4}
	\sup_{a\in\overline B_\mu(V)}\sup_\Sigma
	\tilde r^{-2\alpha-1+j}|\nabla^ju_a(0)|.
	\]
	More precisely, the cutoff derivative estimates give
	\begin{align*}
	K_{\gamma_0,\tau_0}
	\leq{}
    &C_\psi\bigl(e^{-(\lambda_m-\alpha)\tau_0}
	+\gamma_0^{\lambda_m-\alpha}\bigr)
	+C\mu e^{-(\beta-\alpha)\tau_0}
	+C\mu\gamma_0^{\lambda_m-\alpha}
	e^{-(\beta-\lambda_m)\tau_0}
	+Ce^{-c\gamma_0e^{\tau_0}}.
	\end{align*}
	Consequently,
	\[
	\lim_{\gamma_0\downarrow0}\limsup_{\tau_0\to\infty}
	K_{\gamma_0,\tau_0}=0.
	\]
	In particular, the initial graphs satisfy
	the required smallness assumptions.
	
	\medskip
	
	\noindent\emph{Step 2: The refined box.}
	
	Let \(u_a(\tau)\) be the solution starting from
	\eqref{e:initial-graph-prescribe} and set
	\[
	z_a(\tau):=u_a(\tau)-H_\psi(\tau).
	\]
	Define
	\[
	A(\tau)
	:=
	e^{\beta(\tau+\tau_0)}
	\|\Pi_\leq z_a(\tau)\|_{L^2_W}
	\,\text{ and }\,
	B(\tau)
	:=
	e^{\beta(\tau+\tau_0)}
	\|\Pi_>z_a(\tau)\|_{L^2_W}.
	\]
	The refined box is defined by the standard pointwise small graph bounds for \(u_a\) and the additional bounds
	\begin{equation}
	\label{e:refined-box}
	A(\tau)\leq\mu
	\,\text{ and }\,
	B(\tau)\leq\mu.
	\end{equation}
	
	Choose \(\alpha\) outside the spectrum of \(L\) such that
	\begin{equation}
	\label{e:alpha-choice}
	\frac{\beta}{2}<\alpha<\lambda_m.
	\end{equation}
	This is possible because of \eqref{e:beta-choice-prescribe}.  
	Inside the box,
	\[
	\|u_a(\tau)\|_{L^2_W}
	\leq
	\|\psi\|_{L^2_W}e^{-\lambda_m(\tau+\tau_0)}
	+
	2\mu e^{-\beta(\tau+\tau_0)}.
	\]
	Thus
	\[
	\|u_a(\tau)\|_{L^2_W}
	\leq
	c(\tau_0)e^{-\alpha(\tau+\tau_0)},
	\qquad
	c(\tau_0):=\|\psi\|_{L^2_W}e^{-(\lambda_m-\alpha)\tau_0}
	+2\mu e^{-(\beta-\alpha)\tau_0}\to0
	\]
	as $\tau_0\to\infty.$
	The proof of Theorem~\ref{impro} uses the special eigenfunction form
	of its initial data only to obtain the scale-invariant initial bounds.
	Thus the same proof, using
	\eqref{e:E-pointwise-prescribe}--\eqref{e:prescribe-Y3} and the fixed
	cutoff buffer, applies
	with \(c(\tau_0)\) and \(K_{\gamma_0,\tau_0}\).  Consequently, for
	any \(\varepsilon_{\rm pt}>0\), after choosing the parameters as
	above, it gives, uniformly while the solution remains in the box,
	\begin{align}
	|\nabla^\ell u_a|
	&\leq
	\varepsilon_{\rm pt}\tilde r^{2\alpha+1-\ell}
	e^{-\alpha(\tau+\tau_0)}
	\text{ for }
	\ell=0,1,
	\label{e:C2-improvement-1}\\
	|\nabla^2u_a|
	&\leq
	\varepsilon_{\rm pt}\tilde r^{2\alpha}
	e^{-\alpha(\tau+\tau_0)}.
	\label{e:C2-improvement-2}
	\end{align}
	Since \(\operatorname{supp}u_a\subset
	\{f\leq\Gamma_0e^{\tau+\tau_0}\}\), we have
	\(\tilde r^{2\alpha}e^{-\alpha(\tau+\tau_0)}
	\leq C\Gamma_0^\alpha\) there.  Thus choosing
	\(\varepsilon_{\rm pt}\) sufficiently small makes the pointwise box
	inequalities strict.  Hence a first exit cannot occur through a
	pointwise face.
	
	\medskip
	
	\noindent\emph{Step 3: The equation for the remainder.}
	
	Using \eqref{e:forced-graph-prescribe} and \eqref{e:Hpsi-linear},
	and denoting the corresponding transplantation error by
	\(\mathcal E_a\), the remainder \(z_a=u_a-H_\psi\) satisfies
	\begin{equation}
	\label{e:z-equation}
	\partial_\tau z_a
	= Lz_a
	+ P_\tau z_a
	+ R_a,
	\end{equation}
	where
	\begin{equation}
	\begin{split}
	\label{e:Ra-def}
	R_a
	:={}&
	\mathcal E_a
	+Y_0
	+Y_1(u_a,\tilde r\nabla u_a,x,\tau)u_a
	+Y_2^b(u_a,\tilde r\nabla u_a,x,\tau)\nabla_bu_a
	\\
	&+
	\big(
	Y_3^{bc}(u_a,\tilde r\nabla u_a,x,\tau)
	-
	Y^{bc}_{3,0}(x,\tau)
	\big)
	\nabla^2_{bc}u_a
	+
	P_\tau H_\psi
	+
	Q_2(u_a).
	\end{split}
	\end{equation}
	We claim that for any prescribed $\varepsilon_R>0,$ after taking
	$\tau_0$ sufficiently large, we have
	\begin{equation}
	\label{e:R-small}
	e^{\beta(\tau+\tau_0)}
	\|R_a(\tau)\|_{L^2_W}
	\leq
	\varepsilon_R
	\end{equation}
	uniformly while the solution remains in the box.
	We will estimate each of the terms in $R_a.$
	
	Uniformly while the solution remains in the box,
	\eqref{e:E-Gaussian-prescribe} gives, for every \(N\),
	\[
	\|\mathcal E_a(\tau)\|_{L^2_W}
	\leq
	C_Ne^{-N(\tau+\tau_0)}.
	\]
	The quadratic graph error satisfies
	\[
	\|Q_2(u_a)\|_{L^2_W}
	\leq
	C\varepsilon_{\rm pt}^2e^{-2\alpha(\tau+\tau_0)}
	\]
	by \eqref{e:C2-improvement-1}, \eqref{e:C2-improvement-2}, and the Gaussian integrability of polynomial weights.
	The lower order forcing terms satisfy
	\begin{align*}
	\|Y_0\|_{L^2_W}
	&\leq
	C\delta_0e^{-\frac12(\tau+\tau_0)},\quad
	\|Y_1u_a\|_{L^2_W}
	\leq
	C\delta_0e^{-(1+\alpha)(\tau+\tau_0)},\quad
	\|Y_2^b\nabla_bu_a\|_{L^2_W}
	&\leq
	C\delta_0e^{-(\frac12+\alpha)(\tau+\tau_0)}.
	\end{align*}
	By the mean value theorem in the fibre variables, the intervening
	segment being admissible by the strict pointwise improvement, the
	fibre derivative bounds for \(Y_3\),
	\eqref{e:C2-improvement-1}, and \eqref{e:C2-improvement-2}, it follows that
	\[
	\begin{aligned}
	&\left\|
	\big(
	Y_3(u_a,\tilde r\nabla u_a,x,\tau)
	-
	Y_{3,0}(x,\tau)
	\big)\nabla^2u_a
	\right\|_{L^2_W}  
	\leq
	C\delta_0\varepsilon_{\rm pt}^2e^{-2\alpha(\tau+\tau_0)}.
	\end{aligned}
	\]
	Finally, \eqref{e:Ptau-E0-automatic} implies
	\[
	\begin{aligned}
	\|P_\tau H_\psi\|_{L^2_W}
	&=
	e^{-\lambda_m(\tau+\tau_0)}
	\|P_\tau\psi\|_{L^2_W} 
	\leq
	C\eta_0
	e^{-(\lambda_m+1)(\tau+\tau_0)}
	\|\psi\|_{L^2_W}.
	\end{aligned}
	\]
	Because
	\[
	\beta<\frac12,
	\qquad
	\beta<2\alpha,
	\qquad
	\beta<\lambda_m+1,
	\]
	all terms are \(o(e^{-\beta(\tau+\tau_0)})\), uniformly on that
	interval after taking \(\tau_0\) large.
	This proves \eqref{e:R-small}.
	
	\medskip
	
	\noindent\emph{Step 4: Spectral inequalities.}
	
	The compact support of \(u_a\), the polynomial growth of
	\(H_\psi\), and parabolic regularity give
	\(z_a\in C([0,T];H^2_W)\cap C^1([0,T];L^2_W)\) on every compact
	time interval on which the solution remains in the box.
	We use the estimate
	\(
	|\nabla^kY_{3,0}|
	\leq
	C\delta_0\tilde f^{-k/2},\qquad 0\leq k\leq2,
	\)
	which follows from \eqref{e:prescribe-Y3}.
	For \(v\in E_\leq\),
	\begin{equation}
	\label{e:P-low}
	\|P_\tau v\|_{L^2_W}
	\leq
	C\delta_0\|v\|_{L^2_W}
	\end{equation}
	by finite-dimensionality and \eqref{e:finite-spectral-estimates}.
	If \(\{\varphi_q\}\) is an orthonormal basis of \(E_\leq\), then
	\[
	P_\tau^{*,W}\varphi_q
	=e^f\nabla_j\nabla_i(e^{-f}Y_{3,0}^{ij}\varphi_q).
	\]
	Using \(|\nabla f|\leq C\tilde f^{1/2}\),
	\(|\nabla^2f|\leq C\), \eqref{e:prescribe-Y3}, and
	\eqref{e:finite-spectral-estimates}, we obtain
	\(\|P_\tau^{*,W}\varphi_q\|_{L^2_W}\leq C\delta_0\).
	Consequently, for \(w\in E_>\cap H^2_W\),
	\begin{equation}
	\label{e:P-cross}
	\|\Pi_\leq P_\tau w\|_{L^2_W}
	\leq
	C\delta_0\|w\|_{L^2_W}.
	\end{equation}
	
	We write $d\mu_W:=e^{-f}d\mu_\Sigma.$
	For \(w\in E_>\cap H^2_W\), integration by parts and the weighted
	inequality below give
	\begin{equation}
	\label{e:P-high-form}
	|\langle P_\tau w,w\rangle_{L^2_W}|
	\leq
	C\delta_0
	\int_\Sigma
	(|\nabla w|^2+w^2)\,d\mu_W.
	\end{equation}
	The terms involving \(\nabla f\) are controlled by the weighted inequality 
	\[
	\int_\Sigma fw^2\,d\mu_W
	\leq
	C\int_\Sigma(|\nabla w|^2+w^2)\,d\mu_W.
	\]
	This follows from \eqref{shrinker-F-Laplacian} and Cauchy's inequality
	(cf. \cite{BW17, CS}).
	Since
	\[
	\langle Lw,w\rangle_{L^2_W}
	=
	-\int_\Sigma|\nabla w|^2\,d\mu_W
	+
	\int_\Sigma
	\left(|A|^2+\frac12\right)w^2\,d\mu_W
	\]
	and \(|A|\) is bounded, \eqref{e:P-high-form} implies
	\[
	|\langle P_\tau w,w\rangle_{L^2_W}|
	\leq
	C\delta_0
	\left(
	-\langle Lw,w\rangle_{L^2_W}
	+
	\|w\|_{L^2_W}^2
	\right).
	\]
	The same calculation gives
	\[
	|\langle P_\tau v,w\rangle_{L^2_W}|
	\leq C\delta_0\|v\|_{H^1_W}\|w\|_{H^1_W},
	\]
	so the form extends to \(H^1_W\times H^1_W\) and the estimate holds
	on the form domain by density.  Using
	\(\langle Lw,w\rangle_{L^2_W}
	\leq
	-\lambda_+\|w\|_{L^2_W}^2\)
	for $w\in E_>,$
	we obtain
	\begin{equation}
	\label{e:high-dissipative}
	\langle (L+P_\tau)w,w\rangle_{L^2_W}
	\leq
	-(\lambda_+-C\delta_0)\|w\|_{L^2_W}^2.
	\end{equation}
	
	Write
	\[
	z_{\leq}:=\Pi_\leq z_a
	\,\text{ and }\,
	z_>:=\Pi_>z_a.
	\]
	Let \(\underline D^+\) and \(\overline D^+\) denote the lower and
	upper right Dini derivatives.  From \eqref{e:z-equation},
	\eqref{e:R-small}, \eqref{e:P-low}, \eqref{e:P-cross}, and
	\eqref{e:high-dissipative}, norm differentiation gives
	\begin{align}
	\label{e:A-Dini}
	\underline D^+A
	&\geq
	\big(\beta-\lambda_m-C\delta_0\big)A
	-
	C\delta_0B
	-
	e^{\beta(\tau+\tau_0)}
	\|R_a\|_{L^2_W}
	,\\
	\label{e:B-Dini}
	\overline D^+B
	&\leq
	-\big(\lambda_+-\beta-C\delta_0\big)B
	+
	C\delta_0A
	+
	e^{\beta(\tau+\tau_0)}
	\|R_a\|_{L^2_W}.
	\end{align}
	Here, we use that all the eigenvalues on \(E_\leq\) are at most \(\lambda_m\), while all the eigenvalues on \(E_>\) are at least
	\(\lambda_+\).
	
	Set \(g_*:=\min\{\beta-\lambda_m,\lambda_+-\beta\}>0\).  Choose
	\(\delta_0\) and then \(\tau_0\) so that
	\[
	2C\delta_0\leq\frac14g_*,
	\qquad
	\varepsilon_R\leq\frac14g_*\mu.
	\]
	It follows that \(\underline D^+A>0\) on
	\(\{A=\mu,B\leq\mu\}\), while \(\overline D^+B<0\) on
	\(\{B=\mu,A\leq\mu\}\).
	Thus, the \(B\)-face and the pointwise faces are inward pointing,  
	so the only possible first exit from the refined box is through the \(A\)-face, and the crossing there is strictly outward.
	
	\medskip
	
	\noindent\emph{Step 5: The degree argument.}
	
	As long as \(u_a\) remains in the box, parabolic estimates in the
	graphical region and \eqref{e:exterior-regularity} on its complement
	give the bounds required by the continuation criterion.  Thus a
	finite maximal time can occur only through a box exit.  Moreover,
	by \eqref{e:initial-low-projection} and the Gaussian high-mode tail,
	after increasing \(\tau_0\),
	\[
	A(0)=\|a\|_{L^2_W},
	\qquad B(0)<\frac\mu2,
	\]
	and all pointwise faces are strict at time zero.

	For \(a\in\overline B_\mu(V)\), let \(\tau(a)\) be the supremum of
	times on which the solution exists and remains in the box.  Suppose,
	for contradiction, that \(\tau(a)<\infty\) for every \(a\).  By the
	previous step, every first exit is a strictly outward crossing of
	the \(A\)-face.
	Define
	\[
	\Phi(a)
	:=
	e^{\beta(\tau(a)+\tau_0)}
	\Pi_\leq z_a(\tau(a)).
	\]
	Then
	\(\Phi(a)\in\partial B_\mu(V).\)
	If \(a_j\to a\), continuous dependence rules out exit of \(u_{a_j}\)
	before \(\tau(a)-o(1)\), while uniform strict transversality forces
	exit before \(\tau(a)+o(1)\).
	Hence \(\tau(a)\), and therefore \(\Phi\), is continuous.
	
	If \(a\in\partial B_\mu(V)\), then by \eqref{e:initial-low-projection},
	\(e^{\beta\tau_0}\Pi_\leq z_a(0)=a\).  All other faces are strict,
	and \(\underline D^+A(0)>0\); hence \(\tau(a)=0\) and
	\(\Phi(a)=a.\)
	Thus, \(\Phi\) is a continuous retraction from the closed ball
	\(\overline B_\mu(V)\) onto its boundary \(\partial B_\mu(V)\), which is impossible.  
	Hence, there exists \(a_*\in\overline B_\mu(V)\) such that the corresponding solution never exits the box.
	
	For this parameter \(a_*\), we have
	\[
	\|\Pi_\leq z_{a_*}(\tau)\|_{L^2_W}
	+
	\|\Pi_>z_{a_*}(\tau)\|_{L^2_W}
	\leq
	2\mu e^{-\beta(\tau+\tau_0)}.
	\]
	Therefore,
	\(
	\|z_{a_*}(\tau)\|_{L^2_W}
	\leq
	C e^{-\beta(\tau+\tau_0)}.
	\)
	Since
	\(
	u_{a_*}(\tau)
	=
	H_\psi(\tau)+z_{a_*}(\tau),
	\)
	we get
	\[
	\begin{aligned}
	\left\|
	e^{\lambda_m(\tau+\tau_0)}u_{a_*}(\tau)
	-
	\psi
	\right\|_{L^2_W}
	&=
	e^{\lambda_m(\tau+\tau_0)}
	\|z_{a_*}(\tau)\|_{L^2_W} 
	\leq
	C e^{-(\beta-\lambda_m)(\tau+\tau_0)}.
	\end{aligned}
	\]
	Since \(\beta>\lambda_m\), the right hand side tends to zero as \(\tau\to\infty\).  
	This proves the theorem.
\end{proof}

Finally, parabolically rescale the constructed flow and force by
\[
\widehat M_{\widehat t}
:=e^{-\tau_0/2}M_{e^{\tau_0}\widehat t},
\qquad
\widehat\phi[y,\widehat t]:=e^{\tau_0/2}\phi[e^{\tau_0/2}y,e^{\tau_0}\widehat t].
\]
Here \(-e^{-\tau_0}\leq\widehat t<0\).
This replaces \(\tau+\tau_0\) by \(\widehat\tau\) in the
scale-invariant bounds of Assumption~\ref{t:rmcf}.  Since
\(\widehat\tau:=-\log(-\widehat t)=\tau+\tau_0\),
\(\widehat u(\widehat\tau)=u_{a_*}(\widehat\tau-\tau_0)\), and hence
\[
\widehat u(\widehat\tau)
=e^{-\lambda_m\widehat\tau}\psi
+O_{L^2_W}(e^{-\beta\widehat\tau}).
\]
Since the transplanted and actual flows agree on every fixed compact
set for all sufficiently large \(\widehat\tau\), the estimates
\eqref{e:C2-improvement-1}--
\eqref{e:C2-improvement-2} and parabolic interior estimates imply
local smooth convergence to \(\Sigma\).  Moreover, the cutoff from
Step~1 and the Gaussian tail estimate give
\[
\|(1-\chi_{\widehat\tau})\psi\|_{L^2_W}
\leq Ce^{-c\gamma_0e^{\widehat\tau}}\|\psi\|_{L^2_W},
\]
and hence
\[
\|\chi_{\widehat\tau}\widehat u(\widehat\tau)
-e^{-\lambda_m\widehat\tau}\psi\|_{L^2_W}
=o(e^{-\lambda_m\widehat\tau}).
\]
Taking \(\lambda_m=\lambda\), \(\psi=\varphi\), and renaming
\(\widehat\tau\) as \(\tau\) proves
Theorem~\ref{thm:first-order} in the asymptotically conical case.
\section{The case of closed singularities}
\label{sec:closed-case}

Finally, in this section, we deal with closed shrinkers, for which the theorems can be derived by conceptually simpler arguments;
see \cite{Strehlke20, SWX2} for fixed point type arguments in the case of spherical singularities.
We indicate how the arguments in Sections \ref{sec: bar func}--\ref{s:prescribe-first-order} imply the main theorems for a closed shrinker.

For Theorem \ref{thm:forced-realization}, we look at the estimates and arguments in Sections \ref{sec: bar func}--\ref{s:box-forced}.
The calculations in Section \ref{sec: bar func} and Appendix \ref{s:barrier-calc} remain valid after replacing all homogeneous $C^{k}_{{\rm hom},1}$-norms with the standard $C^k$-norms.
In fact, the avoidance principle implies that a closed mean curvature flow stays in a fixed large Euclidean ball, and the homogeneous norms and the usual $C^k$-norms are equivalent.
The barrier functions in Section \ref{sec: bar func} and the rest of the arguments in Section \ref{sec:int-est} then lead to the global $C^2$ improvement theorem for a closed shrinker with all cutoff functions $\eta$ removed.
The box argument (with $\eta$ removed) in Section \ref{s:box-forced} then implies Theorems \ref{thm:riemannian-realization} and \ref{thm:rotational-realization} for such a closed shrinker.

For Theorem \ref{thm:first-order}, we look at Section \ref{s:prescribe-first-order}.
As above, all the cutoff functions are not needed, and we could work with $z_a$ in Step 2 directly without the cutoff estimates in Steps 1 and 3.
The rest of the arguments remain valid and imply Theorem \ref{thm:first-order} for a closed shrinker.

\appendix

\section{Barrier for \texorpdfstring{$\na h$}{nh}} 
\label{s:barrier-calc}

In this section, we define the comparison operator $\tilde{Q}$ and the barrier function $\tilde{u}$ for $\nabla h$, and complete the calculations outlined in Section \ref{sec: bar func}.

Differentiating \eqref{h-equ} and pairing with $\na^g h$ gives
\begin{align*}
&    \partial_\tau|\na^gh|^2\\ \leq&\pr{\Delta^g-\na^g_{\na^g \frac{|F|^2}{4}}} |\na^gh|^2 - 2|\na^{2,g} h|^2
    + 2|A_g|^2|\na^g h|^2+2A^2(\na^g h,\na^g h)
    + C^{ab}\na_a^g\na_b^g|\na^gh|^2\\  
    &+2h\la\na |A|^2,\na h\ra-2C^{ab}g
    ^{cd}\na_{ac}^{2,g}h\na^{2,g}_{bd}h +2|\na^g h\cdot(\na^g \mathcal{Q}_1+\tilde{C}_5h)|
    +|\na^g h|\cdot |\tilde{C}_3*\na^g h+\tilde{C}_4*\na^{2,g}h|.
\end{align*}
The same estimates used for $h$ give, under the smallness assumption
$\|h\|_{C^1_{\rm hom,1}}+\|h\|_{C^2}\leq\ep\leq \ep_n$,
\begin{align*}
    |C^{ab}|\leq C_{n,\Sigma}\pr{|hA|+|\na^g h|^2}\leq& C_{n,\Sigma}\ep,\\
    |\na^gh|\cdot |\tilde{C}_3*\na^g h+\tilde{C}_4*\na^{2,g} h|\leq& {C_{n,\Sigma}}(\ep|\na^{2,g}h|^2+(|A|^2+\ep|\na A|+\frac{\ep}{f})|\na^g h|^2),\\
 |\na^gh|\cdot  |\tilde{C}_5h|\leq&(|A_g|\cdot |\na^{2,g}A_g|+|\na^gA_g|^2)h^2|\na^gh|\leq \frac{C_{n,\Sigma}|\na^gh|\cdot h^2}{f^2}.
\end{align*}
On $\set{(x,\tau)\in \Sigma\times[0,\infty):\Gamma<f(x)<\Gamma e^{\tau+\tau_0}}$, the forcing terms satisfy
\begin{align*}
   & |\na^gh\cdot\na^g\mathcal{Q}_1-(\na^g h)_kY^{ij}_3(\na^3h)_{ijk}|\leq 2\delta_0|\na h|f^{-\frac{1}{2}}(e^{-\frac{1}{2}{(\tau+\tau_0)}}+e^{-\frac{1}{2}{(\tau+\tau_0)}}|\na h|+e^{-{(\tau+\tau_0)}}|h|+|\na^2h|),    \\
        &|(\na^g h)_kY^{ij}_3(\na^3h)_{ijk}-\frac{1}{2}Y^{ij}_3\na^g_i\na^g_j(|\na^gh|^2)|\leq \delta_0(|\na^2h|^2+C_{n}|A|^2|\na^g h|^2).
\end{align*}

The useful point is that $\na^g h$ is governed by the linearized operator
$\partial_\tau-(\Delta_{\frac{|F|^2}{4}}+|A|^2)$ rather than
$\partial_\tau-(\Delta_{\frac{|F|^2}{4}}+|A|^2+\frac{1}{2})$.  This removes one power of $f$ from the gradient barrier.  Assuming
$$h^2\leq u=
        e^{-2\lambda_*(\tau+\tau_0)} \pr{D\tf^{2\lambda_*+1}-B\tf^{2\lambda_*}}
        \quad\text{on }
        \set{(x,\tau)\in \Sigma\times[0,\infty):\Gamma <f(x)<\Gamma e^{\tau+\tau_0}},$$
we define
\begin{align*}
    \tilde{Q}\psi:=&\pr{g^{ab}+C^{ab}+Y_3^{ab}}\na_a^g\na_b^g\psi
    -\na^g_{\na^g \frac{|F|^2}{4}}\psi
    + \pr{C_n|A|^2+C_n|\na A|+\frac{C_n}{\tf}+\delta_0e^{-\frac{1}{2}{(\tau+\tau_0)}}\tf^{-\frac{1}{2}}}\psi
    \\ &+\sqrt{|\psi|}\cdot(2\delta_0e^{-{\frac{1}{2}(\tau+\tau_0)}}\tf^{-\frac{1}{2}}+2\delta_0e^{-{(\tau+\tau_0)}}\tf^{-\frac{1}{2}}\sqrt{u}+C_{n,\Sigma}\tf^{-\frac{3}{2}}\sqrt{u}+C_{n,\Sigma}\tf^{-2}u).\notag
\end{align*}
A parallel barrier calculation gives
\begin{align*}
    \tilde{u}=
        e^{-2\lambda_*(\tau+\tau_0)} \pr{\tilde{D}\tf^{2\lambda_*}-\tilde{B}\tf^{2\lambda_*-1}}
        &\text{ on }
        \set{(x,\tau)\in \Sigma\times[0,\infty):\Gamma <f(x)<\Gamma e^{\tau+\tau_0}},
\end{align*}

For the square-root contribution in $\tilde Q$, we have
\begin{align*}
    & e^{2\lambda_*(\tau+\tau_0)}\sqrt{ e^{-2\lambda_*(\tau+\tau_0)}(\tilde{D}\tf^{2\lambda_*}-\tilde{B}\tf^{2\lambda_*-1})}\pr{2\delta_0e^{-{\frac{1}{2}(\tau+\tau_0)}}\tf^{-\frac{1}{2}}+2\delta_0e^{-{(\tau+\tau_0)}}\tf^{-\frac{1}{2}}\sqrt{u}+C_{n,\Sigma}\tf^{-\frac{3}{2}}\sqrt{u}+C_{n,\Sigma}\tf^{-2}u}\\ 
    \leq &e^{\lambda_*(\tau+\tau_0)}\sqrt{\tilde{D}}\tf^{\lambda_*}\pr{2\delta_0e^{-\frac{1}{2}{(\tau+\tau_0)}}\tf^{-\frac{1}{2}}+2\delta_0e^{-(\lambda_*+1){(\tau+\tau_0)}}\sqrt{{D}}\tf^{\lambda_*}+C_{n,\Sigma}\tf^{-\frac{3}{2}}\sqrt{u}+C_{n,\Sigma}De^{-2\lambda_*(\tau+\tau_0)} \tf^{2\lambda_*-1}}\notag\\ 
    \leq&\sqrt{\tilde{D}}\tf^{2\lambda_*-1}(2\delta_0\Gamma^{\frac{1}{2}-\lambda_*}+2\delta_0\sqrt{D}\Gamma+\sqrt{D}+C_{n,\Sigma}D\Gamma^{\lambda_*}).\notag
\end{align*}
Fix $\Gamma>0$ to be finite, and enlarge $C_{n,\Sigma}$ once and for all
so that it dominates all constants in the preceding drift, potential, and
second-order perturbation estimates.  Prescribe
\[
    \frac{\widetilde B}{\widetilde D}
    =
    C_{n,\Sigma}\left(
        2\lambda_*\left(\frac n2+1\right)
        +\sup_\Sigma\bigl(
            \widetilde f|A|^2
            +\widetilde f|\nabla A|
        \bigr)
        +1
    \right).
\]
Choose $\Gamma$ sufficiently large that
\[
    1+\Gamma
    \geq
    C_{n,\Sigma}
    \max\left\{
        1,\frac{B}{D},
        \frac{\widetilde B}{\widetilde D}
    \right\},
\]
and then choose $\ep$ and $\delta_0$, depending on the fixed $\Gamma$,
sufficiently small that
\[
    C_{n,\Sigma}
    \left(
        \ep+\delta_0
        +\delta_0\Gamma^{1/2}
        +\delta_0\Gamma
    \right)
    \leq\frac14.
\]
For every fixed $T>0$, choose
\begin{equation}\label{e:gradient-BD-choice}
\begin{split}
    \widetilde D=\max\Bigg\{&
        C_{n,\Sigma}^2
        \left(
            \sqrt D
            +2\delta_0\Gamma^{\frac12-\lambda_*}
            +2\delta_0\sqrt D\,\Gamma
            +D\Gamma^{\lambda_*}
        \right)^2,\\
        &2\sup_{\partial_p\mathcal A_{T,\Gamma,\tau_0}}
        e^{2\lambda_*(\tau+\tau_0)}
        \widetilde f(x)^{-2\lambda_*}
        |\nabla h(x,\tau)|^2
    \Bigg\},
\end{split}
\end{equation}
and set
\[
    \widetilde B
    =
    C_{n,\Sigma}\widetilde D
    \left(
        2\lambda_*\left(\frac n2+1\right)
        +\sup_\Sigma\bigl(
            \widetilde f|A|^2
            +\widetilde f|\nabla A|
        \bigr)
        +1
    \right).
\]
These choices imply
\[
    \widetilde u
    =
    e^{-2\lambda_*(\tau+\tau_0)}
    \widetilde f^{2\lambda_*-1}
    \bigl(\widetilde D\widetilde f-\widetilde B\bigr)
    \geq
    \frac{\widetilde D}{2}
    e^{-2\lambda_*(\tau+\tau_0)}
    \widetilde f^{2\lambda_*}>0.
\]
Combining the exact radial calculation with the preceding second-order,
potential, and square-root estimates therefore gives
\[
    (\partial_\tau-\widetilde Q)\widetilde u
    \geq
    \frac{\widetilde B}{8}
    e^{-2\lambda_*(\tau+\tau_0)}
    \widetilde f^{2\lambda_*-1}>0
\quad\text{on}\quad
    \left\{
        (x,\tau)\in\Sigma\times(0,T):
        \Gamma<f(x)<\Gamma e^{\tau+\tau_0}
    \right\}.
\]
The boundary term in \eqref{e:gradient-BD-choice} also guarantees
$|\nabla h|^2\leq\widetilde u$ on the fixed inner face and the moving
outer face.

\section{Computations for the geometric forcing terms}
\label{app:geometric-forcing}

In this appendix, we provide detailed proofs for Propositions \ref{p:riem-admissible} and \ref{p:rot-admissible}.
Recall that we let \(f=|F|^2/4\) and $\td f=f+1,$
and let $\tilde r\geq1$ be smooth with
$\tilde r(x)=|F(x)|+1$ for $|F(x)|\geq1$.

\subsection{The Riemannian manifold forcing}
\label{app:riemannian-forcing}

We prove Proposition~\ref{p:riem-admissible}.  
Choose normal coordinates
\(\xi=(\xi^1,\ldots,\xi^{n+1})\) centered at \(p\), so that
\(g_{\alpha\beta}(0)=\delta_{\alpha\beta}\) and
\(\partial_\gamma g_{\alpha\beta}(0)=0\).  
For \(0<\rho\leq1\), let
\(\iota_\rho(\xi)=\rho\xi\) and set
\[
        g^{(\rho)}:=\rho^{-2}\iota_\rho^*g,
        \qquad
        (g^{(\rho)})_{\alpha\beta}(\xi)
        =g_{\alpha\beta}(\rho\xi).
\]
The rescaled ambient metric is
\[
        \bar g_\tau
        =e^{\tau+\tau_0}
        \iota_{e^{-\frac12(\tau+\tau_0)}}^*g^{(\rho)},
\]
and hence
\[
        (\bar g_\tau)_{\alpha\beta}(\xi)
        =g_{\alpha\beta}
        \bigl(\rho e^{-\frac12(\tau+\tau_0)}\xi\bigr).
\]
After reducing \(\rho\), there is a constant
\(\eta_{\rm amb}=\eta_{\rm amb}(\rho,g,p)>0\), with
\(\eta_{\rm amb}\to0\) as \(\rho\to0\), such that, for
\(2l+|\alpha|\leq4n+8\),
\begin{align}
        |\partial_\tau^l\partial^\alpha(\bar g_\tau-\delta)|(\xi)
        &\leq
        C_{l,\alpha}\eta_{\rm amb}e^{-(\tau+\tau_0)}
        (1+|\xi|)^{2-|\alpha|},
        \label{e:app-riem-metric-est}\\
        |\partial_\tau^l\partial^\alpha\Gamma_{\bar g_\tau}|(\xi)
        &\leq
        C_{l,\alpha}\eta_{\rm amb}e^{-(\tau+\tau_0)}
        (1+|\xi|)^{1-|\alpha|}
        \label{e:app-riem-Christoffel-est}
\end{align}
for \(|\xi|\leq C\Gamma_0e^{\frac12(\tau+\tau_0)}\), provided
\(\rho\Gamma_0\) is sufficiently small.  These estimates follow by
repeatedly differentiating the normal-coordinate expansion of \(g\).

For a hypersurface \(F\colon X\to\bb R^{n+1}\), the scalar forcing term is
\[
        \phi_r[X,\tau]
        =
        \left\langle
        \mathbf H_{\bar g_\tau}(X)-\mathbf H_{\mathrm{euc}}(X),
        \N_X
        \right\rangle_{\mathrm{euc}}.
\]
We apply this to the graph
\[
        X_h=F+h\N_F
\]
of a function $h.$
We temporarily regard the height and the weighted gradient as independent
fibre variables.  Let \(z\) and \(w\) correspond to \(h\) and
\(\tilde r\na h\), respectively, and set
\[
        q=\tilde r^{-1}w,
        \qquad |z|+|w|\leq\ep_0\tilde r.
\]
Set
\[
        X_z=F+z\N_F.
\]
In a local orthonormal frame \(\{e_i\}\) on \(F\), define
\[
        E_i(z,q)
        =dF(e_i)
        +z\, \overline\nabla^{\mathrm{euc}}_{e_i}\N_F
        +q_i\N_F.
\]
For the actual graph,
\[
        z=h,\qquad w=\tilde r\na h,
        \qquad q=\na h.
\]
Whenever \(\na_x\) is applied to a coefficient below, the fibre variables
\((z,w)\) and the time variable are held fixed.
We use throughout the symbol estimates
\[
        |\na^k\tilde r|\leq C_k\tilde r^{1-k},
        \qquad
        |\na^k\tilde r^{-1}|\leq C_k\tilde r^{-1-k}.
\]

Let \(\N_0\) be the Euclidean unit normal to the plane spanned by the
\(E_i\), and let \(P_0^\perp\) be the Euclidean normal projection. 
 Let \(\N_\tau\) be the \(\bar g_\tau\)-unit normal to the same plane, and let \(P_\tau^\perp\) be the \(\bar g_\tau\)-normal projection.  
Define
\(\widehat G_{ij}=\langle E_i,E_j\rangle\)
and \(G^\tau_{ij}=\bar g_\tau(E_i,E_j).\)
The shrinker estimates and the graph smallness assumption imply
\begin{align}
        |\nabla^kX_z|
        &\leq
        C_k\tilde f^{\frac12-\frac{k}{2}},
        \label{e:app-Xz-symbol}\\
        |\nabla^kE_i|
        &\leq
        C_k\tilde f^{-\frac{k}{2}},
        \label{e:app-Ei-symbol}\\
        |\nabla^k\widehat G^{ij}|
        +
        |\nabla^k\N_0|
        +
        |\nabla^kP_0^\perp|
        &\leq
        C_k\tilde f^{-\frac{k}{2}}.
        \label{e:app-euc-symbol}
\end{align}

We now estimate the principal coefficient perturbation
\[
        \mathcal C^{ij}(z,q,x,\tau)
        :=
        (G^\tau)^{ij}P_\tau^\perp
        -
        \widehat G^{ij}P_0^\perp.
\]
Write
\(G^\tau_{ij}
=
\widehat G_{ij}+B_{ij}\)
and \(B_{ij}
=
(\bar g_\tau-\delta)_{X_z}(E_i,E_j).\)
From \eqref{e:app-riem-metric-est}, \eqref{e:app-Xz-symbol}, and \eqref{e:app-Ei-symbol},
\begin{equation}
\label{e:app-Bij-est}
        |\nabla^kB_{ij}|
        \leq
        C_k\eta_{\rm amb} e^{-(\tau+\tau_0)}
        \tilde f^{1-\frac{k}{2}}.
\end{equation}
On the region defined by
\(\tilde f\leq C\Gamma_0^2e^{\tau+\tau_0},\)
we have \(|B_{ij}|\leq C\eta_{\rm amb}\Gamma_0^2.\)
Choosing the coordinate scale so that \(C\eta_{\rm amb}\Gamma_0^2\ll1\), the induced metrics \(G^\tau\) and \(\widehat G\) are uniformly equivalent.  
Since
\[
        (G^\tau)^{-1}-\widehat G^{-1}
        =
        -(G^\tau)^{-1}B\,\widehat G^{-1},
\]
we get
\begin{equation}
\label{e:app-inverse-difference}
        \left|
        \nabla^k
        \big((G^\tau)^{ij}-\widehat G^{ij}\big)
        \right|
        \leq
        C_k\eta_{\rm amb} e^{-(\tau+\tau_0)}
        \tilde f^{1-\frac{k}{2}}.
\end{equation}

The normal comparison is also standard.  
Since \(\N_0\) is Euclidean
orthogonal to \(E_i\),
\[
        \bar g_\tau(\N_0,E_i)
        =
        (\bar g_\tau-\delta)_{X_z}(\N_0,E_i).
\]
The \(\bar g_\tau\)-unit normal is
\[
        \N_\tau
        =
        \mu^{-1/2}
        \left(
        \N_0
        -
        (G^\tau)^{ij}
        \bar g_\tau(\N_0,E_j)E_i
        \right),
\]
where
\[
        \mu
        =
        \bar g_\tau(\N_0,\N_0)
        -
        (G^\tau)^{ij}
        \bar g_\tau(\N_0,E_i)\bar g_\tau(\N_0,E_j).
\]
Using \eqref{e:app-riem-metric-est} and the estimates above, one obtains
\[
        |\nabla^k(\mu-1)|
        \leq
        C_k\eta_{\rm amb} e^{-(\tau+\tau_0)}
        \tilde f^{1-\frac{k}{2}}.
\]
After reducing \(\eta_{\rm amb}\), we have \(\mu\geq1/2\).  
Therefore
\[
        |\nabla^k(\N_\tau-\N_0)|
        \leq
        C_k\eta_{\rm amb} e^{-(\tau+\tau_0)}
        \tilde f^{1-\frac{k}{2}},
\]
and hence
\[
        |\nabla^k(P_\tau^\perp-P_0^\perp)|
        \leq
        C_k\eta_{\rm amb} e^{-(\tau+\tau_0)}
        \tilde f^{1-\frac{k}{2}}.
\]
Combining this with \eqref{e:app-inverse-difference}, we get
\begin{equation}
\label{e:app-principal-est}
        |\nabla^k\mathcal C^{ij}|
        \leq
        C_k\eta_{\rm amb} e^{-(\tau+\tau_0)}
        \tilde f^{1-\frac{k}{2}}.
\end{equation}
The same argument gives the derivative estimates
\begin{align}
        |\nabla^k\partial_z\mathcal C^{ij}|
        &\leq
        C_k\eta_{\rm amb} e^{-(\tau+\tau_0)}
        \tilde f^{\frac12-\frac{k}{2}},
        \label{e:app-principal-z-est}\\
        |\nabla^k\partial_{q_m}\mathcal C^{ij}|
        &\leq
        C_k\eta_{\rm amb} e^{-(\tau+\tau_0)}
        \tilde f^{1-\frac{k}{2}}.
        \label{e:app-principal-q-est}
\end{align}
Indeed, differentiating in \(z\) differentiates either the ambient coefficient
in the normal direction or the graph tensors in the \(z\)-direction, and
therefore loses one factor \(\tilde f^{1/2}\).  Differentiating in \(q\) does not
force this loss.

We now deal with the mean curvature difference.  
With respect to the frame
\(\{e_i\}\),
\begin{align}
        \mathbf H_{\bar g_\tau}(X_h)
        -
        \mathbf H_{\mathrm{euc}}(X_h)
        &=
        \mathcal C^{ij}(h,\nabla h,x,\tau)
        \overline\nabla^{\mathrm{euc}}_{e_i}E_j
        +
        (G^\tau)^{ij}P_\tau^\perp
        \Gamma_{\bar g_\tau}(E_i,E_j).
        \label{e:app-Hdiff}
\end{align}
Moreover,
\begin{equation}
\label{e:app-second-graph}
        \overline\nabla^{\mathrm{euc}}_{e_i}E_j
        =
        (\nabla_i\nabla_jh)\N_F
        +
        \mathcal R_{ij}(h,\nabla h,x),
\end{equation}
where \(\mathcal R_{ij}\) does not depend on \(\nabla^2h\), and
\begin{align}
        |\nabla^k\mathcal R_{ij}|
        &\leq
        C_k\tilde f^{-\frac12-\frac{k}{2}},
        \label{e:app-R-est}\\
        |\nabla^k\partial_z\mathcal R_{ij}|
        &\leq
        C_k\tilde f^{-1-\frac{k}{2}},
        \label{e:app-R-z-est}\\
        |\nabla^k\partial_{q_m}\mathcal R_{ij}|
        &\leq
        C_k\tilde f^{-\frac12-\frac{k}{2}}.
        \label{e:app-R-q-est}
\end{align}

It follows that the $\mathcal{Q}_1$ term
\[
        \frac{\phi_r[F+h\N_F,\tau]}
        {\langle \N_F,\N_{F_r}\rangle}
\]
is affine-linear in \(\nabla^2h\).  Hence, there exist smooth coefficients
\[
        \mathcal A=\mathcal A(z,q,x,\tau),
        \qquad
        \mathcal B^{ij}=\mathcal B^{ij}(z,q,x,\tau),
\]
such that
\begin{equation}
\label{e:app-A-B}
        \frac{\phi_r[F+h\N_F,\tau]}
        {\langle \N_F,\N_{F_r}\rangle}
        =
        \mathcal A(h,\nabla h,x,\tau)
        +
        \mathcal B^{ij}(h,\nabla h,x,\tau)(\nabla^2h)_{ij}.
\end{equation}
The denominator is harmless because
\(\langle\N_F,\N_{F_r}\rangle\geq 1/2.\)

From \eqref{e:app-principal-est}, the coefficients of \(\nabla^2h\) satisfy
\begin{equation}
\label{e:app-B-est}
        |\nabla^k\mathcal B^{ij}|
        \leq
        C_k\eta_{\rm amb} e^{-(\tau+\tau_0)}
        \tilde f^{1-\frac{k}{2}}.
\end{equation}
The lower-order part \(\mathcal A\) has two contributions.  
The principal contribution is \(\mathcal C^{ij}\mathcal R_{ij},\)
which is bounded by
\[
        C_k\eta_{\rm amb} e^{-(\tau+\tau_0)}
        \tilde f^{\frac12-\frac{k}{2}}.
\]
The Christoffel contribution in \eqref{e:app-Hdiff} is controlled by \eqref{e:app-riem-Christoffel-est}.  
Since
\[
        |\Gamma_{\bar g_\tau}(X_z)|
        \leq
        C\eta_{\rm amb} e^{-(\tau+\tau_0)}\tilde f^{1/2},
\]
it satisfies the same bound.  
Thus,
\begin{equation}
\label{e:app-A0-est}
        |\nabla^k\mathcal A(0,0,x,\tau)|
        \leq
        C_k\eta_{\rm amb} e^{-(\tau+\tau_0)}
        \tilde f^{\frac12-\frac{k}{2}}.
\end{equation}
Using \eqref{e:app-principal-z-est},
\eqref{e:app-principal-q-est},
\eqref{e:app-R-z-est}, and \eqref{e:app-R-q-est}, we also get
\begin{align}
        |\nabla^k\partial_z\mathcal A(z,q,x,\tau)|
        &\leq
        C_k\eta_{\rm amb} e^{-(\tau+\tau_0)}
        \tilde f^{-\frac{k}{2}},
        \label{e:app-A-z-est}\\
        |\nabla^k\partial_{q_m}\mathcal A(z,q,x,\tau)|
        &\leq
        C_k\eta_{\rm amb} e^{-(\tau+\tau_0)}
        \tilde f^{\frac12-\frac{k}{2}}.
        \label{e:app-A-q-est}
\end{align}

Define
\begin{align*}
        Y_0(x,s)
        &:=
        \mathcal A(0,0,x,s),
        &\quad
        &Y_1(z,w,x,s)
        :=
        \int_0^1
        \partial_z\mathcal A
        (\theta z,\theta\tilde r^{-1}w,x,s)\,d\theta,
        \\
        Y_2^m(z,w,x,s)
        &:=
        \int_0^1
        \partial_{q_m}\mathcal A
        (\theta z,\theta\tilde r^{-1}w,x,s)\,d\theta,
        &\quad
        &Y_3^{ij}(z,w,x,s)
        :=
        \mathcal B^{ij}(z,\tilde r^{-1}w,x,s).
\end{align*}
By the fundamental theorem of calculus,
\[
\begin{aligned}
        \mathcal A(h,\nabla h,x,\tau)
        &=
        \mathcal A(0,0,x,\tau)  
        +
        \left(
        \int_0^1\partial_z\mathcal A(sh,s\nabla h,x,\tau)\,ds
        \right)h+
        \left(
        \int_0^1\partial_{q_m}\mathcal A(sh,s\nabla h,x,\tau)\,ds
        \right)(\nabla h)_m.
\end{aligned}
\]
Therefore, \eqref{e:app-A-B} implies \eqref{e:forced nonlinear}.
Repeated differentiation of the preceding identities, with the fibre
variables fixed under \(\na_x\), gives, after evaluation at
\((z,w,s)=(h,\tilde r\na h,\tau)\),
\begin{align*}
        |\partial_s^l\na_x^kY_0|
        &\leq
        C_{l,k}\eta_{\rm amb}e^{-(\tau+\tau_0)}
        \tilde f^{\frac12-\frac{k}{2}},
        &\quad
        &\tilde r^\beta
        |\partial_s^l\na_x^k\na_{z,w}^\beta Y_1|
        \leq
        C_{l,k,\beta}\eta_{\rm amb}e^{-(\tau+\tau_0)}
        \tilde f^{-\frac{k}{2}},\\
        \tilde r^\beta
        |\partial_s^l\na_x^k\na_{z,w}^\beta Y_2|
        &\leq
        C_{l,k,\beta}\eta_{\rm amb}e^{-(\tau+\tau_0)}
        \tilde f^{\frac12-\frac{k}{2}},
        &\quad
        &\tilde r^\beta
        |\partial_s^l\na_x^k\na_{z,w}^\beta Y_3|
        \leq
        C_{l,k,\beta}\eta_{\rm amb}e^{-(\tau+\tau_0)}
        \tilde f^{1-\frac{k}{2}}.
\end{align*}
These estimates hold whenever \(2l+k+\beta\leq4n+6\).

For \(|F(x)|\leq\Gamma_0e^{\frac12(\tau+\tau_0)}\), we have
\(\tilde f^{1/2}\leq C\Gamma_0e^{\frac12(\tau+\tau_0)}\).  
Hence
\[
        e^{-(\tau+\tau_0)}\tilde f^{\frac12-\frac{k}{2}}
        \leq
        C\Gamma_0e^{-\frac12(\tau+\tau_0)}\tilde f^{-\frac{k}{2}}\,
\text{ and }\,
        e^{-(\tau+\tau_0)}\tilde f^{1-\frac{k}{2}}
        \leq
        C\Gamma_0^2\tilde f^{-\frac{k}{2}}.
\]
Thus \eqref{e:admissible-Y0}--\eqref{e:admissible-Y3} follow after
choosing the homothety scale so that
\[
        C_{n,\Sigma}\eta_{\rm amb}\Gamma_0^2\leq\delta_0.
\]
This choice also makes the full second-order coefficients uniformly elliptic.

We finally verify the exterior part of Assumption~\ref{t:rmcf}.  Choose
normal-coordinate balls \(U'\Subset U\) so that every capped initial
hypersurface lies in \(U'\), and so that \(\chi(\xi)=|\xi|^2\) satisfies
\(\na_{g^{(\rho)}}^2\chi\geq c g^{(\rho)}\) on \(U\).  Extend
\(g^{(\rho)}|_U\) to a complete bounded-geometry metric \(\widehat g\) on
\(\RR^{n+1}\), with \(\widehat g=g^{(\rho)}\) on \(U\).  Along the flow,
\[
        (\partial_t-\Delta_{M_t})\chi
        =-\operatorname{tr}_{TM_t}\na_{g^{(\rho)}}^2\chi
        \leq-cn.
\]
The maximum principle keeps \(M_t\) in \(U'\).  Hence the same flow is also
a mean curvature flow in \((\RR^{n+1},\widehat g)\).  On
\(f_M\geq\frac14\Gamma_0\), the estimates give
the hypotheses of Chen--Yin pseudolocality
\cite[Theorem~7.3]{CY}.  The pseudolocality and interior derivative
estimates imply, on the part issued from
\(f_M\geq\frac12\Gamma_0\),
\[
        |\na^jA_{M_t}|
        \leq
        C_j\Gamma_0^{-\frac{j+1}{2}},
        \qquad 0\leq j\leq4n+6,
\]
throughout \(-1\leq t<0\).  This proves \eqref{e:exterior-regularity}.
Standard compact mean curvature flow theory supplies the remaining well-posedness and continuation statements.  
This proves Proposition~\ref{p:riem-admissible}.

\subsection{The rotationally symmetric forcing}
\label{app:rotational-forcing}

We prove Proposition~\ref{p:rot-admissible} in this section.  Let
\(m\geq2\), write \(\RR^{n+m}=\RR^n\times\RR^m\), and let
\(SO(m)\) act on the \(\RR^m\)-factor.
The quotient is identified with
\[
        \RR^{n+1}_+=\{(x_1,\ldots,x_n,y):y\geq0\}
\]
where $y:=|\mathbf y|$ for $\mathbf y\in \RR^m.$
Let \(M'\subset\RR^n\times\RR^m\) be an \(SO(m)\)-invariant hypersurface
which does not meet the rotation axis, and let
\(M\subset\RR^{n+1}_+\)
be its quotient.  
Then
\[
        \mathcal A(M)
        := \operatorname{Area}(M')
        = \omega_{m-1} \int_M y^{m-1}\,d\mathcal H^n.
\]
For a normal variation \(X=\zeta\nu\) of \(M\),
\[
\begin{aligned}
        \delta\mathcal A(M)
        &=
        \omega_{m-1}
        \int_M
        \left(
        (m-1)y^{m-2}\zeta\nu_y
        -
        y^{m-1}H_M\zeta
        \right)\,d\mathcal H^n\\
        &=
        -\omega_{m-1}
        \int_M
        \zeta
        \left(
        H_M-\frac{m-1}{y}\nu_y
        \right)
        y^{m-1}\,d\mathcal H^n.
\end{aligned}
\]
Thus, the quotient mean curvature vector is
\[
        \mathbf H_M
        -
        \frac{m-1}{y}(\partial_y)^\perp.
\]

Next, translate the rotation axis to
\(y=-C_{\mathrm{ax}}e^{\frac{\tau_0}{2}}\)
where \(C_{\mathrm{ax}}>1\) will be chosen large.  
For every capped profile and permitted perturbation used below, the cap
construction gives
\[
        \sup_M|y|
        \leq C_\Sigma\Gamma_0^{1/2}.
\]
Thus, after requiring
\(C_{\mathrm{ax}}\geq2C_\Sigma\Gamma_0^{1/2}\), the entire initial
profile satisfies
\begin{equation}
\label{e:app-global-axis-separation}
        C_{\mathrm{ax}}e^{\tau_0/2}+y
        \geq\frac12C_{\mathrm{ax}}e^{\tau_0/2}>0.
\end{equation}
Then the flow becomes
\[
        \partial_tX
        =
        \h_M
        -
        \frac{m-1}{C_{\mathrm{ax}}e^{\frac{\tau_0}{2}}+y}
        (\partial_y)^\perp.
\]
After rescaling, for \(X_h=F+h\N_F\), the forcing term is
\begin{equation}
\label{e:app-rot-force-before}
        \phi_r[X_h,\tau]
        =
        -(m-1)
        e^{-\frac12(\tau+\tau_0)}
        \frac{
        \la \partial_y,\N_{X_h}\ra
        }
        {C_{\mathrm{ax}}+e^{-\frac12(\tau+\tau_0)}
        \la X_h,\partial_y\ra}.
\end{equation}
The graph normal identity is
\begin{equation}
\label{e:app-rot-normal-identity}
        \frac{\la \partial_y,\N_{X_h}\ra}
        {\la\N_F,\N_{X_h}\ra}
        =\la \N_F,\partial_y\ra
        - \la\partial_iF,\partial_y\ra
        \big((I+hA_F)^{-1}\big)^{ij}
        (\na h)_j.
\end{equation}
Combining \eqref{e:app-rot-force-before} and \eqref{e:app-rot-normal-identity}, we get
\begin{equation}
\label{e:app-rot-force-graph}
\begin{aligned}
        \frac{\phi_r[F+h\N_F,\tau]}
        {\la\N_F,\N_{X_h}\ra}
        &=
        -(m-1)
        e^{-\frac12(\tau+\tau_0)}
        \frac{
        \la \N_F,\partial_y\ra
        -
        \la\partial_iF,\partial_y\ra
        \big((I+hA_F)^{-1}\big)^{ij}
        (\na h)_j
        }
        {
        C_{\mathrm{ax}}
        +
        e^{-\frac12(\tau+\tau_0)}
        \la F+h\N_F,\partial_y\ra
        }.
\end{aligned}
\end{equation}

Let \(a\) and \(p\) be independent fibre variables corresponding to
\(h\) and \(\tilde r\na h\), respectively, and assume
\[
        |a|+|p|\leq\ep_0\tilde r(x).
\]
After reducing \(\ep_0\), the matrix \(I+aA_F\) is uniformly invertible.
Here \(s\geq0\) denotes the independent time variable of the coefficient
functions.  
For
\(|F(x)|\leq\Gamma_0e^{\frac12(s+\tau_0)}\), choose
\(C_{\mathrm{ax}}\geq C_\Sigma\Gamma_0\) so that
\begin{equation}
\label{e:app-rot-denom-lower}
        C_{\mathrm{ax}}
        +
        e^{-\frac12(s+\tau_0)}
        \la F+a\N_F,\partial_y\ra
        \geq
        \frac12C_{\mathrm{ax}}.
\end{equation}
Indeed,
\(|\la F+a\N_F,\partial_y\ra|
\leq C\tilde f^{1/2}\)
and
\(e^{-\frac12(s+\tau_0)}\tilde f^{1/2}
\leq C\Gamma_0\)
on the region $\{(x,\tau)\,|\,\,|F(x)|\leq\Gamma_0e^{\frac12(s+\tau_0)}\}$.
Since
\[
\begin{aligned}
&C_{\mathrm{ax}}
+
e^{-\frac12(\tau+\tau_0)}
\la F+h\N_F,\partial_y\ra  
=
C_{\mathrm{ax}}
+
e^{-\frac12(\tau+\tau_0)}
\la F,\partial_y\ra
+
e^{-\frac12(\tau+\tau_0)}
h\la \N_F,\partial_y\ra,
\end{aligned}
\]
we have
\[
\begin{aligned}
&-
e^{-\frac12(\tau+\tau_0)}
\frac{\la\N_F,\partial_y\ra}
{
C_{\mathrm{ax}}
+
e^{-\frac12(\tau+\tau_0)}
\la F+h\N_F,\partial_y\ra
} +
e^{-\frac12(\tau+\tau_0)}
\frac{\la\N_F,\partial_y\ra}
{
C_{\mathrm{ax}}
+
e^{-\frac12(\tau+\tau_0)}
\la F,\partial_y\ra
} \\
=&
e^{-(\tau+\tau_0)}
\frac{\la\N_F,\partial_y\ra^2}
{
\left(
C_{\mathrm{ax}}
+
e^{-\frac12(\tau+\tau_0)}
\la F+h\N_F,\partial_y\ra
\right)
\left(
C_{\mathrm{ax}}
+
e^{-\frac12(\tau+\tau_0)}
\la F,\partial_y\ra
\right)
}
h.
\end{aligned}
\]
Therefore, \eqref{e:app-rot-force-graph} is exactly of the form \eqref{e:forced nonlinear} with
\begin{align}
        Y_0(x,s)
        &:=
        -(m-1)
        e^{-\frac12(s+\tau_0)}
        \frac{\la\N_F,\partial_y\ra}
        {
        C_{\mathrm{ax}}
        +
        e^{-\frac12(s+\tau_0)}
        \la F,\partial_y\ra
        },
        \label{e:app-Y0-rot}\\
        Y_1(a,p,x,s)
        &:=
        (m-1)
        e^{-(s+\tau_0)}
        \frac{\la\N_F,\partial_y\ra^2}
        {
        \left(
        C_{\mathrm{ax}}
        +
        e^{-\frac12(s+\tau_0)}
        \la F+a\N_F,\partial_y\ra
        \right)
        \left(
        C_{\mathrm{ax}}
        +
        e^{-\frac12(s+\tau_0)}
        \la F,\partial_y\ra
        \right)
        },
        \label{e:app-Y1-rot}\\
        Y_2^j(a,p,x,s)
        &:=
        (m-1)
        e^{-\frac12(s+\tau_0)}
        \frac{
        \la\partial_iF,\partial_y\ra
        \big((I+aA_F)^{-1}\big)^{ij}
        }
        {
        C_{\mathrm{ax}}
        +
        e^{-\frac12(s+\tau_0)}
        \la F+a\N_F,\partial_y\ra
        },
        \label{e:app-Y2-rot}\\
        Y_3^{ij}(a,p,x,s)
        &:=
        0.
        \label{e:app-Y3-rot}
\end{align}
These coefficients are independent of \(p\).  Evaluating at
\((a,p,s)=(h,\tilde r\na h,\tau)\) gives
\eqref{e:app-rot-force-graph}, and hence~\eqref{e:forced nonlinear}.

It remains to estimate the coefficients.  All \(x\)-derivatives below are
taken with \((a,p,s)\) fixed.  The shrinker estimates imply
\begin{align}
        |\na_x^k\la\N_F,\partial_y\ra|
        &\leq
        C_k\tilde f^{-\frac{k}{2}},
        \label{e:app-rot-symbol-a}\\
        \tilde r^\beta
        \left|
        \na_x^k\partial_a^\beta
        \left[
        \la\partial_iF,\partial_y\ra
        \big((I+aA_F)^{-1}\big)^{ij}
        \right]
        \right|
        &\leq
        C_{k,\beta}\tilde f^{-\frac{k}{2}},
        \label{e:app-rot-symbol-b}\\
        \tilde r^\beta
        |\na_x^k\partial_a^\beta
        \la F+a\N_F,\partial_y\ra|
        +|\na_x^k\la F,\partial_y\ra|
        &\leq
        C_{k,\beta}\tilde f^{\frac12-\frac{k}{2}}.
        \label{e:app-rot-symbol-denom}
\end{align}
Using \eqref{e:app-rot-denom-lower} and repeated differentiation of an
inverse function, both inverse denominators are bounded by \(2C_{\mathrm{ax}}^{-1}\).  
Every nonzero mixed derivative satisfies 
\begin{equation}
\label{e:app-rot-inverse-est}
        \tilde r^\beta
        \left|
        \partial_s^l\na_x^k\partial_a^\beta
        \left(
        C_{\mathrm{ax}}
        +e^{-\frac12(s+\tau_0)}
        \la F+a\N_F,\partial_y\ra
        \right)^{-1}
        \right|
        \leq
        C_{l,k,\beta}C_{\mathrm{ax}}^{-2}
        e^{-\frac12(s+\tau_0)}
        \tilde f^{\frac12-\frac{k}{2}},
\end{equation}
and the same estimate holds with \(a=0\).  
Here \(l+k+\beta\geq1\).

Combining these estimates with
\eqref{e:app-Y0-rot}--\eqref{e:app-Y3-rot}, we obtain, whenever
\(2l+k+\beta\leq4n+6\),
\begin{align*}
        |\partial_s^l\na_x^kY_0|
        &\leq
        C_{l,k}(m-1)C_{\mathrm{ax}}^{-1}
        e^{-\frac12(s+\tau_0)}
        \tilde f^{-\frac{k}{2}},
        &\quad
        &\tilde r^\beta
        |\partial_s^l\na_x^k\na_{a,p}^\beta Y_1|
        \leq
        C_{l,k,\beta}(m-1)C_{\mathrm{ax}}^{-2}
        e^{-(s+\tau_0)}
        \tilde f^{-\frac{k}{2}},
        \label{e:app-Y1-rot-est}\\
        \tilde r^\beta
        |\partial_s^l\na_x^k\na_{a,p}^\beta Y_2|
        &\leq
        C_{l,k,\beta}(m-1)C_{\mathrm{ax}}^{-1}
        e^{-\frac12(s+\tau_0)}
        \tilde f^{-\frac{k}{2}},
       &\quad
        &\tilde r^\beta
        |\partial_s^l\na_x^k\na_{a,p}^\beta Y_3|
        =0.
\end{align*}
After increasing
\(
        C_{\mathrm{ax}}
        =C_{\mathrm{ax}}(\delta_0,\Gamma_0,m,\Sigma),
\)
these are precisely
\eqref{e:admissible-Y0}--\eqref{e:admissible-Y3}.

We finally verify the exterior-regularity requirement.  
The full lift in $\bb R^{n+1}$ is
\[
        M'_t
        = \left\{
        \bigl(x,(C_{\mathrm{ax}}e^{\tau_0/2}+y)\omega\bigr):
        (x,y)\in M_t,\ \omega\in S^{m-1}
        \right\}.
\]
Let \(\varrho=|\mathbf y|=C_{\mathrm{ax}}e^{\tau_0/2}+y\) denote its distance to the rotation axis \(\RR^n\times\{0\}\).  
By
\eqref{e:app-global-axis-separation},
\[
        \varrho|_{M'_{-1}}
        \geq\frac12C_{\mathrm{ax}}e^{\tau_0/2}.
\]
Consider the shrinking cylinders $\RR^n\times S^{m-1}(r(t))$ with
\[
        r(t)^2
        = \frac{C_{\mathrm{ax}}^2e^{\tau_0}}{16}
        -2(m-1)(t+1).
\]
If \(C_{\mathrm{ax}}^2e^{\tau_0}\geq64(m-1)\), the avoidance principle gives
\[
        \varrho|_{M'_t}\geq r(t)
        \geq\frac{C_{\mathrm{ax}}e^{\tau_0/2}}{4\sqrt2}>0
\]
for every time $t\in[-1,0)$ at which the lifted flow is smooth.  
Thus the lift cannot meet the rotation axis.

If \(\nu_y=\la\N_M,\partial_y\ra\), then
\[
        |A_{M'}|^2
        =|A_M|^2+(m-1)\frac{\nu_y^2}{\varrho^2}.
\]
The higher-order cap estimates and the preceding lower bound for
\(\varrho\) give, on the lift of
\(\set{f_M\geq\frac14\Gamma_0}\),
\[
        |\na^jA_{M'_{-1}}|
        \leq
        C_j\Gamma_0^{-\frac{j+1}{2}},
        \qquad 0\leq j\leq4n+6,
\]
together with a graphical radius comparable to
\(\Gamma_0e^{1/2}\).  Since \(M'_t\) is an unforced
Euclidean mean curvature flow, pseudolocality and the interior derivative
estimates give, on the lift of the part issued from
\(\set{f_M\geq\frac12\Gamma_0}\),
\[
        |\na^jA_{M'_t}|
        \leq
        C_j\Gamma_0^{-\frac{j+1}{2}},
        \qquad 0\leq j\leq4n+6,
\]
throughout \(-1\leq t<0\).  Restricting to the profile directions gives
\eqref{e:exterior-regularity} for the quotient flow.  
Standard compact mean curvature flow theory and preservation of the \(SO(m)\)-symmetry give
the remaining well-posedness and continuation statements.  This proves
Proposition~\ref{p:rot-admissible}.

\section{Capping off asymptotically conical shrinkers}
\label{sec:Cap}

In this section, given an asymptotically conical hypersurface $\Sigma,$ we construct a closed embedded hypersurface by directly capping off the asymptotically conical ends of $\Sigma$.
For any large enough $R,$ the resulting closed hypersurface $\Sigma^{\rm cap}_R$ will be identical to $\Sigma$ in $B_R$ and have small curvature outside $B_{2R};$
that is, $\Sigma^{\rm cap}_R$ satisfies \eqref{cap-agreement} and \eqref{cap-curv-decay}.
Unlike the doubling construction performed in~\cite{LZ}, this procedure does not require a second copy of $\Sigma$ or a reflection across a hyperplane.

Let $\mathcal C$ be the asymptotic cone of $\Sigma$, and denote its link by
\(\Lambda:=\mathcal C\cap S^n.\)
Thus, $\Lambda$ is a smooth closed embedded hypersurface of $S^n$.
For simplicity, assume that $\Lambda$ is connected, as different connected components can be treated separately. 
By the Jordan--Brouwer separation theorem, $S^n\setminus\Lambda$ has two
connected components. Fix one of them and denote its closure by
$\Omega$ so that
\(\partial\Omega=\Lambda.\)

By the asymptotically conical assumption, together with Lu Wang's uniqueness of the asymptotically conical ends of shrinkers \cite{W14}, for $R_0$ large enough, the end of $\Sigma$ can be parametrized over
$\mathcal C\setminus B_{R_0}$ as
\[
    F_0(r,\theta)
    =
    r\theta+f(r,\theta)\,\N_{\mathcal C}(\theta)
    \,\,\text{ for }\,\,
    (r,\theta)\in[R_0,\infty)\times\Lambda,
\]
where $\N_{\mathcal C}$ is a unit normal vector field along
$\mathcal C$ and $f$ satisfies the usual decay estimates; see
\cite{W14,CS}.

Fix $R\gg R_0$. Choose a smooth cutoff function
$\chi_R\colon [R,\infty)\to[0,1]$ satisfying
\[
    \chi_R(r)=
    \begin{cases}
        1, & R\leq r\leq \frac32R,\\
        0, & r\geq 2R,
    \end{cases}
\]
with \(\abs{\partial_r^k\chi_R}\leq C_kR^{-k}\)
for every $k\geq 1$. 
Define
\[
    F_R(r,\theta)
    =
    r\theta
    +
    \chi_R(r)f(r,\theta)\N_{\mathcal C}(\theta)
    \quad\text{for }\,
    (r,\theta)\in[R,3R]\times\Lambda.
\]
Thus, \(F_R(r,\theta)=F_0(r,\theta)\)
when \(R\leq r\leq \frac32R\),
whereas \(F_R(r,\theta)=r\theta\)
when \(2R\leq r\leq 3R.\)
For $R$ sufficiently large, the decay estimates for $f$ imply that $F_R$ is an embedding. 
We may therefore define the truncated hypersurface
\[
    \Sigma_R^{\mathrm{tr}}
    :=
    \pr{\Sigma\setminus F_0((R,\infty)\times \Gamma)}
    \cup
    F_R\big([R,3R]\times\Lambda\big).
\]
Since $\chi_R$ is identically one near $r=R$, the two pieces agree smoothly along $F_0(R,\Lambda)$. 
Moreover,
\[
    \partial\Sigma_R^{\mathrm{tr}}
    = 3R\, \Lambda
    = \{3R\,\omega:\omega\in\Lambda\}.
\]

We now attach a cap to this boundary. 
Consider the scaled spherical domain
\[
    K_R
    :=
    3R\, \Omega
    = \{3R\,\omega:\omega\in\Omega\}
    \subset \partial B_{3R}.
\]
Its boundary is
\(\partial K_R=3R\,\Lambda
= \partial\Sigma_R^{\mathrm{tr}}.\)
Consequently,
\(\Sigma_R^{\mathrm{tr}}\cup K_R\)
is a compact embedded hypersurface with a corner
along $3R\,\Lambda$. 

We next smooth this corner.
Let $\N_\Lambda$ be the unit normal of $\Lambda$ in $S^n$ pointing into $\Omega$. 
For a sufficiently small constant $\varepsilon_0>0$, a neighborhood of $3R\,\Lambda$ can be parametrized
by
\[
    \Psi_R(\theta,s,t)
    =
    (3R+s)
    \exp_\theta^{S^n}
    \left(
        \frac{t}{3R}\N_\Lambda(\theta)
    \right),
\]
where
\(\theta\in\Lambda, |s|<\varepsilon_0R,\)
and \(|t|<\varepsilon_0R.\)
Equivalently, in these coordinates,
\[
    \Psi_R(\theta,s,t)
    =
    (3R+s)
    \left(
        \cos\left(\frac{t}{3R}\right)\theta
        +
        \sin\left(\frac{t}{3R}\right)
        \N_\Lambda(\theta)
    \right).
\]
The conical part of $\Sigma_R^{\mathrm{tr}}$ is represented by
\(
    \{t=0,\ s\leq 0\},
\)
while the spherical cap $K_R$ is represented by
\(
    \{s=0,\ t\geq 0\}.
\)
Thus, in the $(s,t)$-plane, their union is locally modeled on the
broken curve
\[
    \big((-\infty,0]\times\{0\}\big)
    \cup
    \big(\{0\}\times[0,\infty)\big).
\]
Choose a smooth embedded curve
\(\gamma_R\subset\{s\leq 0,\ t\geq 0\}\)
which agrees with the negative $s$-axis and the positive $t$-axis outside the disk
\(\{s^2+t^2<(\varepsilon_0R)^2\},\)
and which smoothly rounds off the corner at $(0,0)$. 
We replace the cornered portion of
$\Sigma_R^{\mathrm{tr}}\cup K_R$ by
\[
    \left\{
        \Psi_R(\theta,s,t):
        \theta\in\Lambda,\ (s,t)\in\gamma_R
    \right\}.
\]
Denote the resulting hypersurface by
\(\Sigma_R^{\mathrm{cap}}.\)

Since $\Psi_R$ is an embedding on the above tubular neighborhood and $\gamma_R$ is a smooth embedded curve, the resulting hypersurface $\Sigma_R^{\mathrm{cap}}$ is smooth and embedded. 
It is compact and has no boundary. 
Furthermore,
\begin{align}\label{cap-agreement}
    \Sigma_R^{\mathrm{cap}}\cap\ovl B_R
    = \Sigma\cap\ovl B_R.
\end{align}
We have therefore constructed a closed embedded hypersurface which agrees with the original shrinker on an arbitrarily large compact set by capping off its asymptotically conical end.

Finally, if the rounding curve $\gamma_R$ is obtained by scaling a fixed rounding profile by the factor $R$, then the geometry of the added cap satisfies the scale-invariant estimate
\begin{align}\label{cap-curv-decay}
    \sup_{\Sigma_R^{\mathrm{cap}}
    \setminus B_{2R}} |\na^j A|
    \leq \frac{C_j}{R^{j+1}}.
\end{align}
Here $C<\infty$ depends only on the link $\Lambda$ and the chosen rounding profile.
Thus, given $\Sigma,$ we could choose such a constant $C=C(\Sigma)<\infty$ such that \eqref{cap-curv-decay} holds for any $R$ large enough.

\end{document}